\documentclass{cmslatex}
\usepackage[paperwidth=7in, paperheight=10in, margin=.875in]{geometry}
 \usepackage[backref,colorlinks,linkcolor=red,anchorcolor=green,citecolor=blue]{hyperref}
\usepackage{amsfonts,amssymb}
\usepackage{amsmath}
\usepackage{cite}
\usepackage{enumerate}
\usepackage{indentfirst}
\usepackage{dsfont}
\usepackage{graphicx, subfig}
\usepackage[ruled,vlined]{algorithm2e}

\renewcommand{\theequation}{\arabic{section}.\arabic{equation}}

   \allowdisplaybreaks
\begin{document}
 \title{Approximate Primal-Dual Fixed-Point based  Langevin  Algorithms for Non-smooth Convex Potentials\thanks{Received date, and accepted date (The correct dates will be entered by the editor).}}


          \author{Ziruo Cai\thanks{Shanghai Jiao Tong University, (sjtu\_caiziruo@sjtu.edu.cn).}
          \and Jinglai Li\thanks{University of Birmingham, (j.li.10@bham.ac.uk).}
          \and Xiaoqun Zhang\thanks{Shanghai Jiao Tong University, (xqzhang@sjtu.edu.cn).}}

         \pagestyle{myheadings} \markboth{Approximate Primal-Dual Fixed-Point based  Langevin  Algorithms}{Ziruo Cai, Jinglai Li and Xiaoqun Zhang} \maketitle

          \begin{abstract}
        
        The Langevin algorithms are frequently used to sample  the posterior 
        distributions in Bayesian inference. 
        In many practical problems, however, the posterior distributions 
         often consist of non-differentiable components, posing
        challenges for the standard Langevin algorithms, as 
        they require to evaluate the gradient of the energy function in each iteration. 
        To this end, a popular remedy is to utilize the proximity operator, 
        and as a result one needs to solve a proximity subproblem in each iteration.
        The conventional practice is to solve the subproblems accurately, which 
        can be exceedingly expensive, as the subproblem needs to be solved in each iteration.  
        We propose an approximate primal-dual fixed-point algorithm for solving the subproblem, which only seeks an approximate solution of the subproblem
        and therefore reduces the computational cost considerably. We provide theoretical analysis of the proposed method and also demonstrate its performance with numerical examples.  
          \end{abstract}


\section{Introduction}\label{intro}

The Bayesian inference approach has become increasingly popular as a tool for 
solving inverse problems~\cite{kaipio2006statistical,tarantola2005inverse}, largely due to 
its ability to quantify the uncertainty in the results. 
Simply put the Bayesian approach casts the sought parameter as a random variable and 
computes a posterior probability distribution of it, conditional on the data observed. 
The ability to accurately and efficiently compute the posterior distribution is crucial for the implementation of the Bayesian framework in real-world problems. 
A common practice to compute the posterior distributions is to generate samples from them, via some sampling schemes, 
such as the Markov Chain Monte Carlo (MCMC) methods~\cite{gilks1995markov}. 
To this end, the Langevin algorithm based Monte Carlo (LMC) methods \cite{ermak1975computer, parisi1981correlation, neal1992bayesian, roberts1996exponential} attract significant attention, 
mainly due to its ability to efficiently explore the state space. 
Loosely speaking, 
the Langevin algorithm consists of the following steps: it first constructs a Langevin system with the target distribution as its invariant measure, 
numerically solves the resulting Langevin system with random initial conditions  for sufficiently long time,
and regard the final states as samples drawn from the target distribution.
In particular the Langevin systems are usually solved with the Euler-Maruyama discretization, 
yielding a sampling scheme analogous to the gradient descent method for optimization. 
The algorithm can be incorporated into a MCMC framework by adding a Metropolis-Hasting accept-reject step, resulting the so-called Metropolis-adjusted Langevin algorithm (MALA) \cite{roberts1996exponential, roberts1998optimal, robert1999monte, roberts2002langevin, xifara2014langevin, dwivedi2018log};
as a contrast, the Langevin algorithms without the Metropolis adjustment is usually referred to as 
the unadjusted Langevin algorithms (ULA). 
We consider both types of Langevin algorithms in this work. 

The  standard Langevin algorithms require to evaluate the gradient of the energy function associate to the target distribution in 
each iteration. 
In many Bayesian inverse problems, however, non-differentiable prior distributions are often used -- a notable example is the Total Variation (TV)
prior used in image reconstruction problems. 
In such problems the posteriors are also not differentiable,
which renders 
the standard LMC algorithms infeasible.  
Considerably efforts have been devoted to developing Langevin algorithms  for non-differentiable distributions \cite{pereyra2016proximal, durmus2018efficient,  salim2019stochastic, chatterji2020langevin, lehec2021langevin, luu2021sampling, mou2022efficient, lau2022bregman} 

Among the existing methods, a very popular class of methods borrow ideas from the non-smooth optimization research,  
constructing an approximation of the actual target distribution,
 and as a result a convex proximity subproblem is solved in each iteration \cite{pereyra2016proximal,durmus2018efficient, luu2021sampling, mou2022efficient}. This idea has been used 
in both Metropolis adjusted and the unadjusted algorithms. 
The computational cost of these methods is typically much higher than the standard Langevin algorithms, as they
require to solve a proximity subproblem in each iteration.
In this regard, it is of critical importance to  improve the efficiency in solving the subproblem.
This work is devoted to addressing the issue and our approach has the following two main ingredients.
First, we adopt  the primal-dual fixed point (PDFP) method developed in~\cite{chen2013primal} for non-smooth convex composite potentials $ U(x) = f(x) + g(Bx) $ to solve the proximity subproblem.
Simply speaking, PDFP solves a non-smooth optimization problem using the primal-dual formulation, and it has been shown that the method has many desired theoretical and computational properties in \cite{chen2013primal, zhu2020stochastic}. 
More importantly, we propose  that, it may not be necessary to solve the subproblem accurately as is usually done in 
the existing methods; rather an approximate solution obtained by conducting a small number of optimization iterations  
may suffice for the sampling accuracy while reducing the computational cost considerably.
We study the strategy via both theoretical analysis and  numerical experiments. 
Theoretically we provide analysis of the sampling error due to the finite-step subproblem optimization. 
Via numerical experiments, we demonstrate that the approximate PDFP (i.e., that with a small number of optimization iterations) based Langevin algorithms,
especially the Metropolis-adjusted version, have very competitive performance in terms of sampling efficiency.

The rest of the paper is organized as follows: Section~\ref{sec:standard} reviews
the standard Langevin algorithms for smooth distributions.
Section~\ref{sec:nonsmooth} considers sampling non-smooth distributions and introduces the proximal MALA (PMALA) approach in particular. 
We present the approximate PDFP based Langevin algorithm in Section~\ref{sec:PDFP_LMC} and provide
a nonasymptotic  error analysis of it in Section~\ref{sec:convergence}. 
Two numerical examples are provided in Section~\ref{sec:examples} to demonstrate the performance of the proposed methods 
and finally Section~\ref{sec:conclusions} concludes the paper. 

\section{The standard Langevin algorithms}\label{sec:standard}

We start with a brief introduction to the standard Langevin algorithms for differentiable energy functions. 
Our goal here is to draw samples from a probability density in the form of 
\begin{equation}
\label{target}
    \pi(\theta) \propto \exp\left( -U(\theta) \right) ,\quad \theta\in \mathbb{R}^d,
\end{equation}
where $U(\theta)$ is the energy function. 
Throughout this work we assume that the energy function $ U(\theta) $ is   convex and  lower semi-continuous,
which is an essential presumption for many theoretical studies. 
Note here that the normalizing constant of $\pi$ in Eq.~\eqref{target} is usually not available in practice, and as such
the sampling methods should not require the knowledge of it.

Assuming $U(\theta)$ is differentiable, we can write down the following Langevin system: 
\begin{equation}
    \label{Langevin diffusion}
    \begin{aligned}
    \text{d} L_t &=  \nabla \log\pi(L_t) \text{d}t + \sqrt{2} \text{d} W_t \\
    &= - \nabla U(L_t) \text{d}t + \sqrt{2} \text{d} W_t,
    \end{aligned}
\end{equation}
where $W_t$ is a standard Wiener process. 
It  should be clear that  $\pi$ is the invariant distribution of process $L_t$.  
Apply the Euler-Maruyama discretization to Eq.~\eqref{Langevin diffusion} , and we obtain the ULA update:
\begin{algorithm}
\KwResult{$ \{ \theta_{n} \}_{n = 1}^{N}  $. }
    $L_U:$ Lipschitz constant of $\nabla U$. \\
    Set $\delta\in (0, 1 / L_U]$,   $ \theta_0\in \mathbb{R}^d $. \\
    \For{$n = 0$ to $N-1$} {
    $\theta_{n+1} = \theta_{n} - \delta \nabla U\left( \theta_{n} \right)+\sqrt{2\delta} \xi_{n}, \quad \xi_{n} \sim \mathcal{N}(0, I).$
    }
 \caption{ULA} \label{ULA}
\end{algorithm}

The choice of $\delta$ is given by \cite{dalalyan2017theoretical} that an upper bound of $\delta$ related to the Lipschitz continuity of $\nabla U$ should imply the convergence of ULA, and an upper bound of $\mathbb{E}U(\theta_n)$.

To remove the bias of ULA, a popular adjustment is to add a Metropolis step to this ULA \cite{robert1999monte, roberts2002langevin, roberts1998optimal}, resulting in the following procedure:
\begin{itemize}
\item Propose a new state by ULA: $\displaystyle Y_{n+1}= \theta_{n} - \delta \nabla U\left( \theta_{n} \right)+\sqrt{2\delta} \xi_{n}, \quad \xi_{n} \sim \mathcal{N}(0, I)$.
\item       Compute acceptance rate:
        $\displaystyle A(Y_{n+1}, \theta_{n}) = 
        \min\left(1, \dfrac{\pi(Y_{n+1})}{\pi(\theta_{n})}\cdot \dfrac{p(\theta_{n}| Y_{n+1})}{p(Y_{n+1}|\theta_{n}) } \right)$\\
        $\displaystyle 
        = \min\left(1, \dfrac{\pi(Y_{n+1})}{\pi(\theta_{n})}\cdot  
        \dfrac{  \exp\left( - \dfrac{1}{4\delta} \left\|  \theta_{n} - Y_{n+1} - \delta\nabla U(Y_{n+1}) \right\|_2^2 \right) }
        {  \exp\left( - \dfrac{1}{4\delta} \left\|  Y_{n + 1} -   \theta_{n} - \delta\nabla U(\theta_{n}) \right\|_2^2 \right) }  
        \right)  $.
\item     Draw $a\sim U[0, 1]$.
        
   \item    If  $a < A(Y_{n+1}, \theta_{n})$;
let $\theta_{n+1} = Y_{n+1}$;
otherwise,       let $\theta_{n+1} = \theta_{n}$.
\end{itemize}

The theoretical properties of the ULA have been extensively studied. Provided that one can have access to the accurate gradient $\nabla U$, the nonasymptotic analysis on convergence and errors is given in \cite{dalalyan2017theoretical} for strongly convex $U$ and \cite{durmus2017nonasymptotic} for convex $U$. Moreover,   \cite{durmus2019analysis} studies the problem  in the convex optimization perspective, by separately considering the gradient descent step and the random walk step in the ULA iteration. When the accurate evaluations of the gradient $\nabla U$ are not available, \cite{dalalyan2019user} investigates the case of using inaccurate gradient when $U$ is strongly convex.
Many techniques and results provided in \cite{dalalyan2017theoretical} will be used here in our theoretical analysis. 

\section{Langevin algorithms for non-smooth distributions}\label{sec:nonsmooth}

In many real-world applications  the energy function $U$ includes some non-differentiable terms.
Obviously the ULA and the MALA algorithms introduced in Section \ref{sec:standard} can not be used directly in this case. 
A straightforward solution is to use the subgradient of $U(\cdot)$ in such problems,
but the algorithm becomes significantly  inefficient compared to smooth distribution as is demonstrated in \cite{pereyra2016proximal}.
In this section we will discuss a proximal Langevin algorithm framework \cite{pereyra2016proximal,durmus2018efficient} for non-differentiable energy functions.

\subsection{Definitions and Propositions}\label{sec:defin}

We first provide some definitions and lemmas that are used in the rest of this work, all of which can be found in \cite{boyd2004convex, bauschke2011convex}.
\begin{definition}
\label{proximal operator definition}
The proximity operator $\operatorname{prox}_f(x):\mathbb{R}^d \rightarrow \mathbb{R}^d $ of function $f:\mathbb{R}^d \rightarrow \mathbb{R}$ is defined by
\begin{equation}
    \operatorname{prox}_{f}(x):= \arg\min_{u\in \mathbb{R}^d}\left(  \dfrac{\|u - x\|^{2}}{2} + f(u) \right).
\end{equation}
\end{definition}

\begin{definition}
\label{nonexpansive operator}
An operator $T:\mathbb{R}^d \rightarrow \mathbb{R}^d $ is firmly nonexpansive if and only if 
\[
\left\|  Tx - Ty  \right\|_2^2 \leqslant \left\langle Tx - Ty, x - y \right\rangle ,\quad \forall x, y\in \mathbb{R}^d.
\]
\end{definition}

\begin{definition}
\label{strongly convex definition}
Let $m \in \mathbb{R}, m > 0$. A  function $f$ is $m$-strongly convex if only if
\begin{equation}
    f(y) \geqslant f(x) + \langle \nabla f(x), y - x\rangle + \dfrac{m}{2} \left\|y - x  \right\|_2^2,\quad \forall x,y\in \mathbb{R}^d.
\end{equation}
\end{definition}

\begin{lemma}
\label{strongly convex lemma1}
Let $m \in \mathbb{R}, m > 0$. If $f$ is $m$-strongly convex, then 
\begin{equation}
    \langle x - y, \nabla f(x) - \nabla f(y) \rangle \geqslant m\| x - y  \|_2^2, \quad \forall x, y\in \mathbb{R}^d.
\end{equation}
\end{lemma}

\begin{lemma}
\label{strongly convex lemma}
Let $m \in \mathbb{R}, m > 0$. Function $h(x)$ is $m$-strongly convex if and only if $h(x) - \dfrac{m}{2}\left\| x \right\|_2^2 $ is convex.
\end{lemma}

\begin{lemma}
\label{prox firmly nonexpansive}
For convex function $f:\mathbb{R}^d \rightarrow \mathbb{R}$,  $\operatorname{prox}_f$ and $I - \operatorname{prox}_f$ are firmly nonexpansive operators.
\end{lemma}
\begin{definition}
\label{Lipschitz gradient definition}
Function $f$ has $M$-Lipschitz continuous gradient if
\begin{equation}
    \| \nabla f(x) - \nabla f(y)  \|_2 \leqslant M \| x - y   \|_2, \quad \forall x, y\in \mathbb{R}^d.
\end{equation}
\end{definition}

\begin{lemma}
\label{Lipschitz gradient lemma}
If $f$ has $M$-Lipschitz continuous gradient, then
\begin{equation}
f(y) \leqslant f(x) + \langle \nabla f(x), y - x\rangle + \dfrac{M}{2} \left\| y - x  \right\|_2^2,\quad \forall x,y\in \mathbb{R}^d.    
\end{equation}
Moreover, if $f$ is convex, then
\begin{equation}
M\langle x - y, \nabla f(x) - \nabla f(y) \rangle \geqslant  \|  \nabla f(x) - \nabla f(y)\|_2^2 , \quad \forall x, y\in \mathbb{R}^d.
\end{equation}
\end{lemma}

\begin{definition}
\label{conjugate function}
The conjugate function of function $g$ is defined by
\begin{equation}
    g^*(v) = \sup_{y\in \text{dom}(g)} \left(v^T y - g(y)  \right),
\end{equation}
where $ v\in V =  \text{dom}(g^*) = \{  v| g^*(v) < \infty \} $.
\end{definition}

\subsection{Proximal Langevin algorithms}

%

To tackle this situation when $U$ is convex but non-smooth, \cite{pereyra2016proximal} and \cite{durmus2018efficient, pereyra2020accelerating} respectively replace the original $\pi$ with two continuously differentiable distributions which can be arbitrarily close to $\pi$. In this work we follow the Moreau approximation settings in \cite{pereyra2016proximal}, for any $ \rho > 0 $, define the $ \rho- $Moreau approximation of $ \pi $ as
\begin{equation}
    \label{Moreau approximation}
    \pi_{\rho}(\theta) = \dfrac{1}{K'}\sup_{u\in \mathbb{R}^{d}}\left(\pi(u)\exp\left( - \dfrac{\| u - \theta \|^{2}}{2\rho}   \right)\right).
\end{equation}
By simple computation,
\begin{equation}
    \label{Moreau approximation 2}
    \pi_\rho(\theta)  =    \dfrac{1}{K'} \exp\left(  -U_\rho(\theta)  \right),
\end{equation}
where 
\[
U_\rho(\theta) = \min_{y\in \mathbb{R}^d}\left( U(y) + \dfrac{\left\| y - \theta \right\|_2^2}{2\rho} \right) = U(\operatorname{prox}_{\rho U}(\theta)) + \dfrac{\| \operatorname{prox}_{\rho U}(\theta) - \theta \|^2}{2\rho}
\]
is the Moreau envelope \cite{moreau1965proximity} of $U(\theta)$. By \cite{bauschke2011convex, combettes2005signal}, $ \pi_\rho $ and $U_\rho$ have several useful properties summarized in Lemma \ref{Moreau lemma}:
\begin{lemma}
\label{Moreau lemma}
(1) When $ \rho\rightarrow 0  $, $ \pi_{\rho}(\theta) \rightarrow \pi(\theta) $ pointwisely and $ U_{\rho}(\theta) \rightarrow U(\theta) $ pointwisely. \\
(2) $U_\rho(\theta)$ is convex and has $\frac{1}{\rho}-$Lipschitz continuous gradient. \\
(3) $U(\theta)$ and $U_\rho(\theta)$ have the same minimizers. \\
(4) Even $\pi$ and $U$ can be non-differentiable, $\pi_\rho$ and $U_\rho$ are continuously differentiable and 
\begin{equation}
    \nabla U_\rho(\theta)   = \dfrac{\theta - \operatorname{prox}_{\rho U}(\theta)  }{\rho}.
\end{equation}
\end{lemma}
Replace the original $\pi$ with $\pi_\rho$ in Langevin diffusion (\ref{Langevin diffusion}) and one obtains the SDE
\begin{equation}
    \label{Langevin diffusion Moreau}
    \begin{aligned}
        \text{d} L_t^\rho &=  \nabla \log \pi_\rho(L_t^\rho) \text{d}t + \sqrt{2} \text{d} W_t\\
        &= - \nabla U_\rho(L_t^\rho) \text{d}t + \sqrt{2} \text{d} W_t.
    \end{aligned}
\end{equation}
Here the solution $L_t^\rho\rightarrow \pi_\rho$ in TV norm as $t\rightarrow +\infty$ from Lemma \ref{pi_rho exponentially fast lemma} (see also Lemma 1 in \cite{dalalyan2017theoretical}). By Euler-Maruyama discretization and  Lemma \ref{Moreau lemma} (4) one obtains the proximal ULA \cite{pereyra2016proximal}:
\begin{equation}
    \begin{aligned}
    \theta_{n+1}&=\theta_{n} - \delta \nabla U_{\rho} \left(\theta_{n}\right)+\sqrt{2\delta} \xi_{n}  \\
    & =  \theta_{n} - \delta\frac{\theta_n - \operatorname{prox}_{\rho U}\left(\theta_{n}\right)}{\rho} +\sqrt{2\delta} \xi_{n}, \quad \xi_{n} \sim \mathcal{N}(0, I) \\
    & =\left(1-\frac{\delta}{ \rho}\right) \theta_{n}+\frac{\delta}{\rho} \operatorname{prox}_{\rho U}\left(\theta_{n}\right)+\sqrt{2\delta} \xi_{n}, \quad \xi_{n} \sim \mathcal{N}(0, I).
    \end{aligned}
\end{equation}

\begin{algorithm}
\KwResult{$ \{ \theta_{n} \}_{n = 1}^{N}  $. }
    Set $ \rho > 0$, $\delta\in (0, \rho]$,   $ \theta_0\in \mathbb{R}^d $. \\
    \For{$n = 0$ to $N-1$} {
    $\displaystyle \theta_{n+1}= \left(1-\frac{\delta}{ \rho}\right) \theta_{n}+\frac{\delta}{\rho} \operatorname{prox}_{\rho U}\left(\theta_{n}\right)+\sqrt{2\delta} \xi_{n}, \quad \xi_{n} \sim \mathcal{N}(0, I)$
    }
 \caption{Proximal ULA} \label{proximalULA}
\end{algorithm}
Basically $\rho$ is the parameter of the Moreau approximation and $\delta$ is the stepsize of the Euler-Maruyama discretization, therefore $\delta$ should be independent of $\rho$. For the stability of the algorithm $\delta$ should be within  $ (0, \rho]$ (Proposition 1 in \cite{dalalyan2017theoretical}), and \cite{pereyra2016proximal} sets $\delta = \rho$ yielding a more concise algorithm:
\begin{equation}
    \Rightarrow \theta_{n+1}=\operatorname{prox}_{\delta h} \left(\theta_{n}\right)+\sqrt{2\delta} \xi_{n}, \quad \xi_{n} \sim \mathcal{N}(0, I).
\end{equation}
However, for discretization error analysis one should fix $\rho$ and let $\delta\rightarrow 0$. In this work, we do not constrain $\delta$ to be equal to $\rho$ and in the later section we denote proximal ULA by Algorithm \ref{proximalULA}.

So far, proximal ULA (Algorithm \ref{proximalULA}) has introduced two errors to draw samples from $\pi$: one is the Moreau approximation error from $\pi$ to $\pi_\rho$, another is the discretization error from Langevin diffusion (\ref{Langevin diffusion Moreau}) to Algorithm \ref{proximalULA}. One can eliminate these errors by adding a Metropolis-Hasting accept-reject step \cite{robert1999monte, roberts1998optimal, roberts2002langevin, roberts1996exponential} and turns proximal ULA into proximal Metropolis-Adjusted Langevin Algorithm (MALA) \cite{pereyra2016proximal}: 

\begin{algorithm}[ht!]

\KwResult{$ \{ \theta_{n} \}_{n = 1}^{N}  $. }
    Set $ \rho > 0$, $\delta\in (0, \rho]$,  $ \theta_0\in \mathbb{R}^d $. \\
    \For{$n = 0$ to $N-1$} {
        \If{$n > 0$}{ 
        Compute $\operatorname{prox}_{\rho U}\left(\theta_{n}\right)$ according to the previous accept-reject step: \\
        $\operatorname{prox}_{\rho U}\left(\theta_{n}\right) = \operatorname{prox}_{\rho U}\left(\theta_{n - 1}\right)$, or $\operatorname{prox}_{\rho U}\left(\theta_{n}\right) = \operatorname{prox}_{\rho U}\left( Y_{n} \right)$
        }
        Propose a new state by proximal ULA: $\displaystyle Y_{n+1}= \left(1-\frac{\delta}{ \rho}\right) \theta_{n}+\frac{\delta}{\rho} \operatorname{prox}_{\rho U}\left(\theta_{n}\right)+\sqrt{2\delta} \xi_{n}, \quad \xi_{n} \sim \mathcal{N}(0, I)$\\
        Compute acceptance rate:
        $\displaystyle A(Y_{n+1}, \theta_{n}) = 
        \min\left(1, \dfrac{\pi(Y_{n+1})}{\pi(\theta_{n})}\cdot \dfrac{p(\theta_{n}| Y_{n+1})}{p(Y_{n+1}|\theta_{n}) } \right)$\\
        $\displaystyle 
        = \min\left(1, \dfrac{\pi(Y_{n+1})}{\pi(\theta_{n})}\cdot  
        \dfrac{  \exp\left( - \dfrac{1}{4\delta} \left\|  \theta_{n} - \left(1-\frac{\delta}{ \rho}\right) Y_{n+1}-\frac{\delta}{\rho} \operatorname{prox}_{\rho U}\left(Y_{n+1}\right) \right\|_2^2 \right) }
        {  \exp\left( - \dfrac{1}{4\delta} \left\|  Y_{n + 1} -  \left(1-\frac{\delta}{ \rho}\right) \theta_{n}-\frac{\delta}{\rho}\operatorname{prox}_{\rho U}\left(\theta_{n}\right) \right\|_2^2 \right) }  
        \right)  $
        \\
        Sample $a$ from uniform distribution: $a\sim U[0, 1]$.
        \\
        \eIf{$a < A(Y_{n+1}, \theta_{n})$}{
        Accept $Y_{n+1}$: $\theta_{n+1} = Y_{n+1}$
        }{
        Reject $Y_{n+1}$: $\theta_{n+1} = \theta_{n}$
    } 
}
 \caption{Proximal MALA}\label{proximal MALA}
\end{algorithm}
From above, $\{ \theta_{n} \}_{n = 1}^N  $ in Algorithm \ref{proximal MALA} is actually a Metropolis-Hastings markov chain proposed by proximal ULA. Noted that the information of the Moreau approximation $\pi_\rho$ is included in the proposal step but in the accept-reject step $\pi$ is evaluated instead. For efficient computation, in the $n$-th iteration we need to know both $ \operatorname{prox}_{\rho U}\left(\theta_{n}\right) $ and $\operatorname{prox}_{\rho U}\left(Y_{n+1}\right) $, but actually only $ \operatorname{prox}_{\rho U}\left(Y_{n+1}\right) $ need to be computed since $ \operatorname{prox}_{\rho U}\left(\theta_{n}\right) $ can be obtained from the $(n - 1)$-th iteration: If $\theta_n = Y_n$ then $\operatorname{prox}_{\rho U}\left(\theta_{n}\right) = \operatorname{prox}_{\rho U}\left(Y_{n}\right)$, which has been computed in the $(n - 1)$-th iteration. If $\theta_n = \theta_{n - 1}$ then $\operatorname{prox}_{\rho U}\left(\theta_{n}\right) = \operatorname{prox}_{\rho U}\left(\theta_{n-1}\right)$. The challenge is, 
 each iteration of the sampling scheme  involves solving an optimization problem $\operatorname{prox}_{\rho U}\left(x\right)$. 
 In both PMALA \cite{pereyra2016proximal} and MYULA \cite{pereyra2020accelerating}, the algorithm of \cite{chambolle2004algorithm} is used to solve
 the subproblem, and in next section we will discuss an alternative method for this.

\section{The approximate PDFP based Langevin Algorithms}\label{sec:PDFP_LMC}
\subsection{The primal-dual fixed point algorithm}\label{sec:PDFP}

Before considering drawing samples from the given distribution $\pi(\theta)$, in this section we introduce the Primal-Dual Fixed Point (PDFP) algorithm developed in \cite{chen2013primal} and some theoretical results of it. 
Here we assume that the energy function $U$ can be decomposed into two parts
\begin{equation}
U(x) = f(x) + g(Bx),\label{e:fg}
\end{equation}
where 
\begin{itemize}
\item $f$ is convex and continuously differentiable with a $M_2-$Lipschitz gradient. 
\item $B$ is a linear operator.
\item $g$ is convex and perhaps non-differentiable but having a proximity operator $ \operatorname{prox}_g(x) $ which is generally easy to compute.
\end{itemize}
Note here that Eq.~\eqref{e:fg} is a very generic form of posterior distributions in Bayesian inference. 

Recall the following convex minimization problem which can be understood as computing a point estimator by maximizing the posterior distribution:
\begin{equation}
    \label{MAP}
    \min_{x\in \mathbb{R}^d} ~ f(x) + g(Bx).
\end{equation}
Alternatively, Eq.~\eqref{MAP} can be reformulated as  
 a min-max problem 
\begin{equation}
    \label{minmax reformulation}
    \min_{x\in \mathbb{R}^d} \max_{v\in V} ~ f(x) + \langle  Bx, v \rangle - g^*(v).
\end{equation}
Both problem (\ref{MAP}) and its min-max reformulation (\ref{minmax reformulation}) have been well studied in the last decades, e.g., \cite{esser2010general, chambolle2011first}.
The PDFP method (detailed in Algorithm~\ref{PDFP for optimization})  
is a fixed point iteration based algorithm to solve the min-max problem~\eqref{minmax reformulation}
and consequently it solves problem~\eqref{MAP} as well.

\begin{algorithm}

\KwResult{$ \{x_k\}_{k = 1}^K  $. }
    Set  
    $ 0 < \lambda \leqslant \dfrac{1}{\lambda_{\max}(BB^T)  } $, 
    $ 0 < \gamma < \dfrac{2}{M_2} $, 
    $  x_0\in \mathbb{R}^d$, $ v_0 \in V $.

    \For{$k = 0$ to $K-1$}{
        $y_{k+1}  =  x_k - \gamma  \nabla f(x_k)  - \gamma  B^T v_k$ \\
        $v_{k+1}  =  \operatorname{prox}_{\frac{\lambda}{\gamma}g^*}\left( \dfrac{\lambda}{\gamma}By_{k+1}  + v_k  \right)$ \\
        $x_{k+1}  =  x_k - \gamma  \nabla f(x_k)  - \gamma B^T v_{k+1}$  
}
 \caption{Primal-Dual Fixed Point method for problem (\ref{MAP})}\label{PDFP for optimization}
\end{algorithm}

As one can see, Algorithm \ref{PDFP for optimization} generates two sequence, the primal variable sequence $ \{x_k\}_{k = 1}^K $ and the dual variable sequence $ \{ v_k \}_{k = 1}^K $. For the min-max problem (\ref{minmax reformulation}), $ x_k $ and $v_k$ will converge to the optimal primal point $ x^* $ and the optimal dual point $ v^* $ respectively. Noted that the convergence of PDFP (Algorithm \ref{PDFP for optimization}) does not require the strongly convexity of $U(x)$, but from Theorem 3.7 in \cite{chen2013primal} one has the linear convergence rate when $f(x)$ is strongly convex and $ \rho_{\min}(BB^T)>0 $.

To simplify the notation, in the $k$-th iteration of Algorithm \ref{PDFP for optimization} one denotes $T_1(v_k, x_k)$ and $T_2(v_k, x_k)$ by 
\begin{equation}
\label{T1 T2 definition}
\left\{
\begin{aligned}
v_{k+1} &= \operatorname{prox}_{\frac{\lambda}{\gamma}g^*}\left( \dfrac{\lambda}{\gamma}B\left(x_k - \gamma  \nabla f(x_k)  - \gamma  B^T v_k\right)  + v_k  \right) &=: T_1(v_k, x_k)\\
x_{k+1} &= x_k - \gamma  \nabla f(x_k)  - \gamma B^T T_1(v_k, x_k) &=: T_2(v_k, x_k).
\end{aligned}
\right.
\end{equation}
Define the operator $T(v, x)$  by 
\begin{equation}
    T(v_k, x_k) := (v_{k+1}, x_{k+1}) = (T_1(v_k, x_k), T_2(v_k, x_k)),
\end{equation}
then one can deduce the fixed point property of PDFP proved in \cite{chen2013primal}:
\begin{lemma}
\label{PDFP fixed point lemma}
$(v^*, x^*)$ is a fixed point of $T$:
\begin{equation}
    (v^*, x^*) = T(v^*, x^*),
\end{equation}
which is 
\begin{equation}
    \left\{
    \begin{aligned}
    v^* &= \operatorname{prox}_{\frac{\lambda}{\gamma}g^*}\left( \dfrac{\lambda}{\gamma}B\left(x^* - \gamma  \nabla f(x^*)\right) +  (I- \lambda B B^T) v^*   \right)  \\
    x^* &= x^* - \gamma  \nabla f(x^*)  - \gamma B^T v^*.
    \end{aligned}
    \right.
\end{equation}
\end{lemma}
Different from Theorem 3.7 in \cite{chen2013primal}, here we give another version of the linear convergence lemma of PDFP. This lemma shows that $x_k\rightarrow x^*$ and $v_k\rightarrow v^*$ simultaneously, but the linear convergence rate is for $(v, x)$ with the norm defined by $ \left\| (v, x)  \right\|_{\frac{\gamma^2}{\lambda}} := \sqrt{\left\| x \right\|_2^2 + \dfrac{\gamma^2}{\lambda}\left\| v \right\|_2^2} $, which means $x_k $ alone does not necessarily converges at a linear rate to $ x^*$ ignoring $v_k$. For the simplification of notation we define $\phi(x) := x - \gamma \nabla  f(x) $, $M := I - \lambda BB^T$. 
\begin{lemma}
\label{PDFP norm lemma}
Assume that $x^*$ and $v^*$ are the optimal solutions of problem (\ref{minmax reformulation}). Assume that $\{x_k\}_k$ and $\{v_k\}_k$ are the two sequences generated by Algorithm \ref{PDFP for optimization}. Assume that $ \gamma, \lambda$ are the parameters in Algorithm \ref{PDFP for optimization}. If   $\rho_{\min}(BB^T) > 0$ and $\exists \eta_1 \in [0, 1)$ such that $\left\| \phi(x) - \phi(y) \right\|_2 \leqslant \eta_1\|x - y\|_2 $, $\forall x, y\in \mathbb{R}^d$, then $\forall k\in \mathbb{N}$,
\begin{equation}
    \left\| x_k - x^* \right\|_2^2 + \dfrac{\gamma^2}{\lambda}\left\| v_k - v^* \right\|_2^2 \leqslant \eta^k \left( \left\| x_0 - x^* \right\|_2^2 + \dfrac{\gamma^2}{\lambda}\left\| v_0 - v^* \right\|_2^2 \right), \quad  0\leqslant \eta < 1,
\end{equation}
where  $ \eta = \max\left( \eta_1^2, 1 - \lambda\rho_{\min} (BB^T)\right) $.
    
\end{lemma}
\begin{proof}
See Appendix~\ref{Appendix_PDFP norm lemma}.

\end{proof}

\textbf{Remark}. If $f$ is $m_f$-strongly convex, then the condition that $\left\| \phi(x) - \phi(y) \right\|_2 \leqslant \eta_1 \left\| x - y \right\|_2 $, $\eta_1 < 1$ is easily satisfied:
\begin{equation}
\label{phi when f strongly convex}
    \begin{aligned}
    &  \|\phi(x) - \phi(y)\|_2^2 =  \left\|  x- y - \gamma \left( \nabla f(x) - \nabla f(y) \right) \right\|_2^2  \\
    & =   \left\|   x - y \right\|_2^2  + \gamma^2  \left\| \nabla f(x) - \nabla f(y) \right\|_2^2 - 2\gamma  \left\langle x - y, \nabla f(x) - \nabla f (y)  \right\rangle \\
    & \leqslant  \left\|   x - y \right\|_2^2 - \left(\dfrac{2\gamma}{M_2} - \gamma^2\right) \left\| \nabla f(x) - \nabla f (y) \right\|_2^2.
    \end{aligned}
\end{equation}
The  inequality follows from the fact that $f$ has $M_2$-Lipschitz gradient and lemma \ref{Lipschitz gradient lemma}. \\
Since $ 0 < \gamma < \dfrac{2}{M_2}$, we have $ \dfrac{2\gamma}{M_2} - \gamma^2  > 0$. From the assumption that $f$ is $m_f$-strongly convex and lemma \ref{strongly convex lemma1}, we have
\begin{equation}
    \begin{aligned}
    &m_f \| x - y  \|_2^2\leqslant  \langle x - y, \nabla f(x) - \nabla f(y) \rangle \leqslant \| x - y\|_2 \| \nabla f(x) - \nabla f(y)\|_2, \quad \forall x, y\in \mathbb{R}^d  \\
    &\Rightarrow m_f \| x - y  \|_2 \leqslant \| \nabla f(x) - \nabla f(y)\|_2, \quad \forall x, y\in \mathbb{R}^d.
    \end{aligned}
\end{equation}
Then from (\ref{phi when f strongly convex}),
\begin{equation}
    \|\phi(x) - \phi(y)\|_2^2 \leqslant \left(1 - m_f^2\left(\dfrac{2\gamma}{M_2} - \gamma^2\right) \right) \| x - y  \|_2^2.
\end{equation}
Therefore $\eta_1 = \sqrt{1 - m_f^2\left(\dfrac{2\gamma}{M_2} - \gamma^2\right)}$ and $\eta_1\in [0, 1)$ since $m_f \leqslant M_2$.

\subsection{K-step PDFP-based Langevin Algorithms}

This subsection discusses how to implement the PDFP based ULA and MALA to sample the distribution density (\ref{Moreau approximation 2}). 
The two algorithms are based on Algorithm \ref{proximalULA} and Algorithm \ref{proximal MALA}
 respectively. 
 Recall that in Algorithms \ref{proximalULA} and \ref{proximal MALA}, an optimization subproblem 
\begin{equation}
    \label{PDFP for prox}
    \operatorname{prox}_{\rho U} \left(\theta_{n}\right) = \arg\min_{x\in \mathbb{R}^d}\left( \dfrac{\| x - \theta_n \|^{2}}{2\rho} + f(x)  + g(Bx)  \right)
\end{equation}
 needs to be solved. The object function in Eq.~(\ref{PDFP for prox}) changes with respect to different $\theta_n$. We then apply the PDFP algorithm to Eq.~\eqref{PDFP for prox}, yielding the following iteration: 
\begin{equation}
\label{PDFP iteration of K-step}
    \left\{
    \begin{aligned}
        y_{n, k+1} &= x_{n, k} - \gamma\left( \nabla f(x_{n, k}) + \dfrac{1}{\rho}(x_{n, k} - \theta_n)\right)  -  \gamma B^T v_{n, k}           \\
        v_{n, k+1} &= \operatorname{prox}_{\frac{\lambda}{\gamma}g^*}\left( \dfrac{\lambda}{\gamma}By_{n, k+1}  + v_{n, k}  \right)  \\
        x_{n, k+1} &= x_{n, k} - \gamma\left( \nabla f(x_{n, k})+  \dfrac{1}{\rho}(x_{n, k} - \theta_n)\right)  -  \gamma B^T v_{n, k+1} .
    \end{aligned}
    \right.
\end{equation}
Inserting the PDFP iteration in Eqs.~\eqref{PDFP iteration of K-step} into Algorithms \ref{proximalULA} and \ref{proximal MALA}, 
yields Algorithms  \ref{ULA by PDFP} (ULA-PDFP) and \ref{MALA by PDFP} (MALA-PDFP) respectively. 

\begin{algorithm}

\KwResult{$ \{ \theta_{n} \}_{n = 1}^{N}  $.  }
    Set $ \rho > 0$, $\delta\in (0,\rho]$, 
    $\displaystyle 0 < \lambda \leqslant \dfrac{1}{\lambda_{\max}(BB^T)  } $, 
    $\displaystyle  0 < \gamma <  \dfrac{2}{M_2 + 1 / \rho} $,
      $\theta_0\in \mathbb{R}^d$.

    \For{$n = 0$ to $N-1$}{ 
        Initialization:  $  x_{n, 0} = \theta_n$, $ v_{n, 0} = 0 $.
        
        \For{$k = 0$ to $K-1$}{
        $y_{n, k+1} = x_{n, k} - \gamma\left( \nabla f(x_{n, k}) + \dfrac{1}{\rho}(x_{n, k} - \theta_n)\right)  -  \gamma B^T v_{n, k}$           \\
        $v_{n, k+1} = \operatorname{prox}_{\frac{\lambda}{\gamma}g^*}\left( \dfrac{\lambda}{\gamma}By_{n, k+1}  + v_{n, k}  \right)$  \\
        $x_{n, k+1} = x_{n, k} - \gamma\left( \nabla f(x_{n, k})+  \dfrac{1}{\rho}(x_{n, k} - \theta_n)\right)  -  \gamma B^T v_{n, k+1}$ 
        }
        $\theta_{n+1} = \left(1-\frac{\delta}{ \rho}\right) \theta_{n}+\frac{\delta}{\rho}x_{n, K} + \sqrt{2\delta} \xi_{n}, \quad \xi_{n} \sim \mathcal{N}(0, I)$ 
}
\caption{ULA-PDFP}\label{ULA by PDFP}
\end{algorithm}

It is natural to ask why we solve Eq.~\eqref{PDFP for prox} by PDFP, instead of other algorithms such as FISTA \cite{beck2009fast} and Chambolle-Pock (CP) \cite{chambolle2011first}. Firstly, FISTA cannot directly solve Eq.~\eqref{MAP} when $B$ is not an identity matrix and solving Eq.~\eqref{PDFP for prox} by FISTA requires a two-layer subproblem. Secondly, solving Eq.~\eqref{PDFP for prox} by CP requires an additional conjugate-gradient algorithm even for $K=1$, which is inefficient when function $f$ includes a non-trivial forward operator. When $f$ is zero and the Moreau envelope is applied merely on $g$, this is what actually MYULA \cite{durmus2018efficient} is doing and therefore CP can solve Eq.~\eqref{PDFP for prox} with the conjugate-gradient algorithm analytically solved. See more details of the experiments between ULA-PDFP and MYULA-CP in Section \ref{sec:examples}.

\begin{algorithm}

\KwResult{$ \{ \theta_{n} \}_{n = 1}^{N}  $. }
    Set $ \rho > 0$, $\delta \in (0, \rho]$, 
    $\displaystyle 0 < \lambda \leqslant \dfrac{1}{\lambda_{\max}(BB^T)  } $, 
    $\displaystyle 0 < \gamma < \dfrac{2}{M_2 + 1 / \rho} $, 
        $P_0 = \theta_0\in \mathbb{R}^d$.
    
    \For{$n = 0$ to $N-1$}{
        Propose a new state:  
        $\displaystyle Y_{n+1}= \left(1-\frac{\delta}{ \rho}\right) \theta_{n}+\frac{\delta}{\rho}P_n +\sqrt{2\delta} \xi_{n}, \quad \xi_{n} \sim \mathcal{N}(0, I)$ \\ ~ \\
        Initialization: $  x_{n, 0} = Y_{n+1} $, $ v_{n, 0} = 0  $
        
        \For{$k = 0$ to $K-1$}{
        $y_{n, k+1} = x_{n, k} - \gamma\left( \nabla f(x_{n, k}) + \dfrac{1}{\rho}(x_{n, k} - Y_{n+1}) \right)  -  \gamma B^T v_{n, k}$           \\
        $v_{n, k+1} = \operatorname{prox}_{\frac{\lambda}{\gamma}g^*}\left( \dfrac{\lambda}{\gamma}By_{n, k+1}  + v_{n, k}  \right)$  \\
        $x_{n, k+1} = x_{n, k} - \gamma\left( \nabla f(x_{n, k})+  \dfrac{1}{\rho}(x_{n, k} - Y_{n+1}) \right)  -  \gamma B^T v_{n, k+1}$ 
        }
        $P_{\text{tmp}} = x_{n, K} $\\ 
        Compute acceptance rate:
        $\displaystyle A(Y_{n+1}, \theta_{n}) = 
        \min\left(1, \dfrac{\pi(Y_{n+1})}{\pi(\theta_{n})}\cdot \dfrac{p(\theta_{n}| Y_{n+1})}{p(Y_{n+1}|\theta_{n}) } \right)$\\
        $\displaystyle 
        = \min\left(1, \dfrac{\pi(Y_{n+1})}{\pi(\theta_{n})}\cdot  
        \dfrac{  \exp\left( - \dfrac{1}{4\delta} \left\|  \theta_{n} -  \left(1-\frac{\delta}{ \rho}\right) Y_{n+1} - \frac{\delta}{\rho}P_{\text{tmp}} \right\|_2^2 \right) }
        {  \exp\left( - \dfrac{1}{4\delta} \left\|  Y_{n + 1} - \left(1-\frac{\delta}{ \rho}\right) \theta_{n} - \frac{\delta}{\rho}P_n \right\|_2^2 \right) }  
        \right)  $
        \\
        Sample $a$ from uniform distribution: $a\sim U[0, 1]$.
        \\
        \eIf{$a < A(Y_{n+1}, \theta_{n})$}{
        Accept $Y_{n+1}$: $\theta_{n+1} = Y_{n+1}$, $P_{n+1} = P_{\text{tmp}}$ 
        }{
        Reject $Y_{n+1}$: $\theta_{n+1} = \theta_{n}$, $P_{n+1} = P_n$
    } 
}
\caption{MALA-PDFP}\label{MALA by PDFP}
\end{algorithm}

Note here that an important feature of the proposed algorithms are that they only conduct a fixed number (i.e., $K$) of PDFP iterations, 
a key difference from the existing algorithms that requires to solve the proximal subproblem  $ \operatorname{prox}_{\rho U} \left(\theta_{n}\right) $ accurately.  
Consequently $x_{n, K}$ is only an approximation of $ \operatorname{prox}_{\rho U} \left(\theta_{n}\right) $ and  Algorithm \ref{ULA by PDFP} is actually an ULA with inaccurate gradient. 
The motivation for doing this is to reduce the computational cost -- as one can see 
each iteration needs to evaluate $\nabla f(x)$, and so the computational cost for computing $ \operatorname{prox}_{\rho U} \left(\theta_{n}\right) $
may be exceedingly high, 
especially when evaluating $\nabla f(x)$ itself is time-consuming. 
In this case, using a small number of iterations (i.e. small value of $K$) may effectively reduce the computational cost. 
Since the approximation is used, the resulting sampling error in Algorithm~\ref{ULA by PDFP} must be analyzed (note that the approximation 
does not introduce sampling error in Algorithm~\ref{MALA by PDFP} thanks to the Metropolis step). 

It should be noted that, in the  iteration in  Algorithms \ref{ULA by PDFP} and \ref{MALA by PDFP} we initialize the dual variable $v_{n, 0} = 0$ instead of  $v_{n, 0} =   v_{n-1, K}$, different from the optimization algorithm. 
The reason is that the Langevin algorithms are expected to generate a Markov Chain $\{\theta_n\}$, which means that the $(n+1)$-th state $\theta_{n+1}$ only depends on the $n$-th state $\theta_n$ and transition probability $P(\theta_{n+1}| \theta_n)$. Once the dual variable $v_{n, 0}$ is initialized as $v_{n-1, K}$, it actually involves the information in the $(n-1)$-th state and the transition probability hence becomes $P(\theta_{n+1}| \theta_n, \theta_{n-1}  )$, violating the Markov property of sequence $\{ \theta_n \}$. 

 Recall that, if $ \operatorname{prox}_{\rho U} \left(\theta_{n}\right) $ is accurately evaluated, then from \cite{dalalyan2017theoretical, durmus2017nonasymptotic, durmus2019analysis} one directly has the convergence and the upper bound on the sampling error of  Algorithm \ref{proximalULA}. 
As has been mentioned, Algorithm \ref{ULA by PDFP} is actually an ULA with inaccurate gradient and so its convergence property needs to be studied. \cite{dalalyan2019user} considers both deterministic and stochastic approximations of the gradient of the log-density and quantifies the impact of the gradient evaluation inaccuracies. In Algorithm \ref{ULA by PDFP}  one intuitively has better upper bound on the sampling error for larger $K$, but at more computational cost. The detailed error analysis is presented in Section~\ref{sec:convergence}.
We also want to mention that, our numerical experiments illustrate that the PDFP based algorithms with small $K$
can produce sufficiently accurate samples, with more details in Section~\ref{sec:examples}.

\section{Convergence results}\label{sec:convergence}

In this section we present the convergence analysis of ULA with $K$-step PDFP (Algorithm \ref{ULA by PDFP}). Most of our proofs follow from \cite{dalalyan2017theoretical}. 
To start with, we first give a lemma which specifies the strongly convexity of the Moreau envelope of a given strongly convex function. 

\begin{lemma}
\label{strongly convex Moreau envelope lemma}
Let $m, \rho \in \mathbb{R}, m > 0, \rho > 0$. If function $h(x)$ is $m$-strongly convex, then the $\rho$-Moreau envelope of $h(x)$,
\begin{equation}
    h_\rho(x) = \min_y\left( h(y) + \dfrac{\left\| y - x \right\|_2^2}{2\rho}  \right),
\end{equation}
is $\dfrac{m}{1+\rho m}$-strongly convex.
\end{lemma}
\begin{proof}
Define $p(x):=h(x) - \dfrac{m}{2}\left\| x \right\|_2^2$. Then from Lemma \ref{strongly convex lemma}, $p(\cdot)$ is convex. By the definition,
\begin{equation}
    \begin{aligned}
    h_\rho(x) &= \min_y\left( h(y) + \dfrac{\left\| y - x \right\|_2^2}{2\rho}  \right) \\
    & = \min_y\left( h(y) - \dfrac{m}{2}\left\| y \right\|_2^2 + \dfrac{m}{2}\left\| y \right\|_2^2 + \dfrac{\left\| y - x \right\|_2^2}{2\rho}  \right) \\
    & = \min_y\left( h(y) - \dfrac{m}{2}\left\| y \right\|_2^2  + \dfrac{(1 + \rho m)}{2\rho}\left\| y - \frac{x}{1 + \rho m} \right\|_2^2 + \dfrac{m}{2(1+\rho m)}\left\| x\right\|_2^2  \right) \\
    & = \min_y\left( p(y)  + \dfrac{(1 + \rho m)}{2\rho}\left\| y - \frac{x}{1 + \rho m} \right\|_2^2 \right) + \dfrac{m}{2(1+\rho m)}\left\| x\right\|_2^2.
    \end{aligned}
\end{equation}
Define $q(z) := p\left(\dfrac{z}{1+\rho m}\right) $. Then function $q(\cdot)$ is convex. \\
From the exchange of variables $y = \dfrac{z}{1+\rho m}$,
\begin{equation}
    \begin{aligned}
    h_\rho(x) &= \min_z\left( q(z)  + \dfrac{1}{2\rho(1+\rho m)}\left\| z - x \right\|_2^2 \right) + \dfrac{m}{2(1+\rho m)}\left\| x\right\|_2^2   \\
    & = q_{\rho(1+\rho m)}(x) + \dfrac{m}{2(1+\rho m)}\left\| x\right\|_2^2.
    \end{aligned}
\end{equation}
$  h_\rho(x) - \dfrac{m}{2(1+\rho m)}\left\| x\right\|_2^2 = q_{\rho(1 + \rho m)}(x) $ is the Moreau envelope of $q$, hence convex.\\
From Lemma \ref{strongly convex lemma}, $h_\rho(x)$ is $\dfrac{m}{1+\rho m}$-strongly convex.
\\
\end{proof}

From Lemma \ref{Moreau lemma}, one can see that  when $\rho \rightarrow 0$, $ h_\rho(x) \rightarrow h(x) $ pointwisely and this result is consistent with $\frac{m}{1 + \rho m} \rightarrow m$. When $\rho \rightarrow +\infty$, $ h_\rho(x) $ tends to a constant function and $\frac{m}{1+\rho m} \rightarrow 0$. In the later convergence analysis of Algorithm \ref{ULA by PDFP} when we  require the strongly convexity of $U_\rho(x)$,   the strong convexity of $U(x)$ is sufficient.

For the study of ULA with inaccurate gradient of log-density, \cite{dalalyan2019user} gives an upper bound of the sampling error when the inaccuracies of the gradients have bounded expectations and variances, with the   assumption that $U$ is strongly convex. Actually the convergence of PDFP (Algorithm \ref{PDFP for optimization}) and convergence of ULA with accurate gradients do not require the strongly convexity of $U$. To prove the boundness of the samples generated by Algorithm \ref{ULA by PDFP}, we  need  the same assumption that $U$ is strongly convex. In this case, we assume that $f$ is $m$-strongly convex and therefore $U_\rho$ is $\frac{m}{1+\rho m}$-strongly convex from Lemma \ref{strongly convex Moreau envelope lemma}.

Another assumption we make is the boundness of $ \left\|\operatorname{prox}_{\frac{\lambda}{\gamma}g^*}(v)\right\|_2  $. This is true when $g$ is the $L^1$ norm and  $g^*$ is an indicator function of a  bounded convex set. 

Since the PDFP iteration and the optimal primal and dual solution of problem (\ref{PDFP for prox}) change with different $\theta_n$, we simplify the notation by denoting the PDFP iteration of problem (\ref{PDFP for prox}) in Algorithm \ref{ULA by PDFP} as 
\begin{equation}
\label{definition of Tn}
\begin{aligned}
&v_{n, k+1} = T_{n, 1}(v_{n, k}, x_{n, k}),\quad x_{n, k+1} = T_{n, 2}(v_{n, k}, x_{n, k}) \\
&\Rightarrow (v_{n, k+1}, x_{n, k+1}) = T_n(v_{n, k}, v_{n, k}) := \left(T_{n, 1}(v_{n, k}, x_{n, k}), T_{n, 2}(v_{n, k}, x_{n, k})\right).
\end{aligned}
\end{equation}
From this notation, the iteration (\ref{PDFP iteration of K-step}) and Algorithm \ref{ULA by PDFP} turns into
\begin{equation}
\label{ULA by PDFP simplified notation}
\left\{
\begin{aligned}
& x_{n, 0} = \theta_n, \quad v_{n, 0} = 0  \\
&v_{n, K} = T_{n, 1}T_{n}^{K - 1}\left( v_{n, 0}, x_{n, 0} \right)  \\
&x_{n, K} = T_{n, 2}T_{n}^{K - 1}\left( v_{n, 0}, x_{n, 0} \right)  \\
&\theta_{n+1} = \theta_{n} - \frac{\delta}{\rho}\left(\theta_n -  x_{n, K}\right) + \sqrt{2\delta} \xi_{n}, \quad \xi_{n} \sim \mathcal{N}(0, I)\\ 
\end{aligned}
\right.
\end{equation}
With a $K$-step PDFP iteration, Algorithm \ref{ULA by PDFP} and (\ref{ULA by PDFP simplified notation}) evaluate the gradient $\nabla U_\rho(\theta_n)$ by the approximation $\dfrac{  \theta_n - T_{n, 2}T_{n}^{K - 1}\left( 0, \theta_{n} \right) }{\rho}$, leading to the error
\begin{equation}
    \nabla U_\rho(\theta_n) - \dfrac{ \theta_n -  T_{n, 2}T_{n}^{K - 1}\left( 0, \theta_{n} \right)}{\rho} = \dfrac{  T_{n, 2}T_{n}^{K-1}\left(0, \theta_n\right) - \operatorname{prox}_{\rho U}(\theta_n)}{\rho}.
\end{equation}
Since the function $ \dfrac{\| x - \theta_n \|^{2}}{2\rho} + f(x) $ is always strongly convex even if $f$ is not strongly convex, we then give a lemma which quantifies the error of $K$-step PDFP in Algorithm \ref{ULA by PDFP}:
\begin{lemma}
\label{subproblem x norm lemma}
Assume that $\{\theta_n\}_n$  is the sequence generated by Algorithm \ref{ULA by PDFP}.   Assume that $  \rho, K, \lambda, \gamma $ are the parameters in Algorithm \ref{ULA by PDFP}. Let $m \geqslant0$, $m\in \mathbb{R}$ . Assume that $f$ is $m$-strongly convex and $\rho_{\min}(BB^T) > 0$. If $ g $ is a function such that, $\forall v\in V$, $  \left\|\operatorname{prox}_{\frac{\lambda}{\gamma}g^*}(v)\right\|_2 \leqslant C$, then $\forall n\in \mathbb{N}$, 
\begin{equation}
    \begin{aligned} 
     \left\| \dfrac{  T_{n, 2}T_{n}^{K-1}\left(0, \theta_n\right) - \operatorname{prox}_{\rho U}(\theta_n)}{\rho} \right\|_2^2 \leqslant \eta^K \left( \left\| \nabla U_\rho\left( \theta_n \right) \right\|_2^2 + \dfrac{\gamma^2C^2}{\lambda\rho^2} \right),
    \end{aligned}
\end{equation}
where 
\begin{equation}
    \eta = \max\left( 1 -\left(m+ \dfrac{1}{\rho}\right)^2\left( \dfrac{2\gamma}{M_2 + \frac{1}{\rho}} - \gamma^2 \right), 1 - \lambda\rho_{\min} (BB^T) \right).
\end{equation}
\end{lemma}
\begin{proof}
See Appendix~\ref{Appendix_subproblem x norm lemma}.

\end{proof} 

To obtain the convergence analysis of Algorithm \ref{ULA by PDFP}, we use the same proof technique in \cite{dalalyan2017theoretical, durmus2019analysis} to first obtain some upper bound of 
$\mathbb{E}\left( U_\rho(\theta_{n+1}) - U_\rho(\theta_n) \right)$. To be more specific, we respectively give the bound of $\mathbb{E}\left( U_\rho\left(  \theta_{n} - \frac{\delta}{\rho}\left(\theta_n -  x_{n, K}\right) + \sqrt{2\delta} \xi_{n}  \right) - U_\rho\left(  \theta_{n} - \frac{\delta}{\rho}\left(\theta_n -  x_{n, K}\right)  \right) \right)$ and 
$\mathbb{E} \left( U_\rho( \theta_{n} - \frac{\delta}{\rho}\left(\theta_n -  x_{n, K}\right)  - U_\rho(\theta_n)  \right)  $, which both simply make use of the Lipschitz gradient of $U_\rho$. Those are explained by Lemma \ref{U_rho lemma 1} and Lemma \ref{U_rho lemma 2}.
\begin{lemma}
\label{U_rho lemma 1}
$\forall x\in \mathbb{R}^d$, if $\xi \sim \mathcal{N}(0, I)$ is independent of $x$, then
\begin{equation}
    \mathbb{E}\left( U_\rho(x + \sqrt{2\delta}\xi) - U_\rho(x) \right)\leqslant \dfrac{\delta d}{\rho}.
\end{equation}
\end{lemma}
\begin{proof}
From Lemma \ref{Moreau lemma} (2), $U_\rho$ has  $\frac{1}{\rho}$-Lipschitz gradient, then by Lemma \ref{Lipschitz gradient lemma},
\begin{equation}
    \begin{aligned}
    U_\rho(x + \sqrt{2\delta}\xi) - U_\rho(x)\leqslant \left\langle \nabla U_\rho(x), \sqrt{2\delta}\xi \right\rangle + \dfrac{1}{2\rho}\left\| \sqrt{2\delta}\xi \right\|_2^2.
    \end{aligned}
\end{equation}
From the assumption that  $\xi \sim \mathcal{N}(0, I)$ is independent of $x$, we have
$
    \mathbb{E}\left\langle \nabla U_\rho(x), \sqrt{2\delta}\xi \right\rangle = 0.
$
\\
Then 
\begin{equation}
    \mathbb{E}\left(U_\rho(x + \sqrt{2\delta}\xi) - U_\rho(x)\right)\leqslant   \dfrac{1}{2\rho}\mathbb{E}\left\| \sqrt{2\delta}\xi \right\|_2^2 = \dfrac{\delta d}{\rho}.
\end{equation}
\end{proof}

\begin{lemma}
\label{U_rho lemma 2}
$\forall x \in \mathbb{R}^d$, $\forall v\in V$,  
\begin{equation}
\small
    \begin{aligned}
    &   U_\rho\left(x - \dfrac{\delta}{\rho}\left(x - T_{n, 2}T_{n}^{K-1}\left(v, x\right)\right)\right) - U_\rho\left(x\right)    \leqslant  -\delta\left(1 - \dfrac{\delta}{2\rho}\right)\left\| \nabla U_\rho(x)\right\|_2^2 \\ 
    &+ \dfrac{\delta^2}{2\rho}\left\| \dfrac{T_{n, 2}T_{n}^{K-1}\left(v, x\right) - \operatorname{prox}_{\rho U}(x)}{\rho} \right\|_2^2 + \delta\left( 1 - \dfrac{\delta}{\rho} \right)\left\langle \nabla U_\rho(x), \dfrac{T_{n, 2}T_{n}^{K-1}\left(v, x\right) - \operatorname{prox}_{\rho U}(x)}{\rho} \right\rangle.
    \end{aligned}
\end{equation}
\end{lemma}
\begin{proof}
From Lemma \ref{Moreau lemma} (2), $U_\rho$ has  $\frac{1}{\rho}$-Lipschitz gradient, then by Lemma \ref{Lipschitz gradient lemma},
\begin{equation}
    \begin{aligned}
        &  U_\rho\left(x - \dfrac{\delta}{\rho}\left(x - T_{n, 2}T_{n}^{K-1}\left(v, x\right)\right)\right) - U_\rho\left(x\right) \\
        & \leqslant - \delta \left\langle \nabla U_\rho(x),  \dfrac{x - T_{n, 2}T_n^{K - 1}(v, x)}{\rho}   \right\rangle + \dfrac{\delta^2}{2\rho} \left\| \dfrac{x - T_{n, 2}T_n^{K - 1}(v, x)}{\rho} \right\|_2^2 \\
        & = - \delta \left\langle \nabla U_\rho(x),  \nabla U_\rho(x) - \dfrac{T_{n, 2}T_{n}^{K-1}\left(v, x\right) - \operatorname{prox}_{\rho U}(x)}{\rho}   \right\rangle \\
        & \quad\quad + \dfrac{\delta^2}{2\rho} \left\| \nabla U_\rho(x) -  \dfrac{T_{n, 2}T_{n}^{K-1}\left(v, x\right) - \operatorname{prox}_{\rho U}(x)}{\rho} \right\|_2^2 \\
        & = -\delta\left(1 - \dfrac{\delta}{2\rho}\right)\left\| \nabla U_\rho(x)\right\|_2^2 + \dfrac{\delta^2}{2\rho}\left\| \dfrac{T_{n, 2}T_{n}^{K-1}\left(v, x\right) - \operatorname{prox}_{\rho U}(x)}{\rho} \right\|_2^2 \\
        & \quad\quad + \delta\left( 1 - \dfrac{\delta}{\rho} \right)\left\langle \nabla U_\rho(x), \dfrac{T_{n, 2}T_{n}^{K-1}\left(v, x\right) - \operatorname{prox}_{\rho U}(x)}{\rho} \right\rangle.
    \end{aligned}
\end{equation}
\end{proof}

From Lemma \ref{U_rho lemma 1} and Lemma \ref{U_rho lemma 2} we can deduce the following lemma showing that the upper bound of $\mathbb{E}\left( U_\rho(\theta_{n+1}) - U_\rho(\theta_n) \right)$ can be controlled by $ \mathbb{E}\left\| \nabla U_\rho(\theta_n) \right\|_2^2 $.
\begin{lemma}
\label{basic lemma for convergence}
Assume that $\{\theta_n\}_n$ is the sequence generated by Algorithm \ref{ULA by PDFP}. Assume that  $x^*$ is the optimal solution of problem (\ref{MAP}). Assume that $\delta, \rho, K, \lambda, \gamma $ are the parameters in Algorithm \ref{ULA by PDFP}. Let $m \geqslant0$, $m\in \mathbb{R}$ . Assume that $f$ is $m$-strongly convex and $\rho_{\min}(BB^T) > 0$. If $ g $ is a function such that, $\forall v\in V$, $  \left\|\operatorname{prox}_{\frac{\lambda}{\gamma}g^*}(v)\right\|_2 \leqslant C$, then $\forall n\in \mathbb{N}$, 
\begin{equation} 
      \mathbb{E} \left( U_\rho\left( \theta_{n+1} \right) - U_\rho\left( \theta_n \right)  \right) \leqslant   - \dfrac{\delta}{2}(1 - \eta^K) \mathbb{E}\left\| \nabla U_\rho(\theta_n) \right\|_2^2 + \dfrac{2\delta d \lambda\rho + \delta\gamma^2C^2\eta^K}{2\lambda\rho^2} , 
\end{equation}
where 
\begin{equation}
    \eta  = \max\left( 1 -\left(m+ \dfrac{1}{\rho}\right)^2\left( \dfrac{2\gamma}{M_2 + \frac{1}{\rho}} - \gamma^2 \right), 1 - \lambda\rho_{\min} (BB^T) \right).
\end{equation}
\end{lemma}
\begin{proof}
See Appendix \ref{Appendix_basic lemma for convergence}.

\end{proof}

In the above lemma, whether $m = 0$ or $m > 0$ simply makes a difference in $\eta$. If $m>0$ we can further   deduce the boundness of $\mathbb{E}U_\rho(\theta_n)$ and $\mathbb{E}\left\| \theta_n - x^* \right\|_2^2  $, by the following theorem:
\begin{theorem}
\label{samples bound lemma}
Under the conditions in Lemma \ref{basic lemma for convergence}, if $m > 0$,  then $\forall n\in \mathbb{N}$, 
\begin{equation}
\label{ineq:upper_bound_Expectation}
    \begin{aligned}
    \mathbb{E}\left(U_\rho\left(\theta_n\right)  - U_\rho\left( x^* \right)\right)\leqslant \left(1 - m_\rho\delta(1 - \eta^K)\right)^n  \mathbb{E}\left( U_\rho(\theta_0) - U_\rho(x^*) \right) + \dfrac{ 2d \lambda\rho +  \gamma^2C^2\eta^K}{2\lambda\rho^2 m_\rho(1 - \eta^K)},
    \end{aligned}
\end{equation}
where  
\begin{equation}
    m_\rho = \dfrac{m}{1+\rho m}, \quad \eta = \max\left( 1 -\left(m+ \dfrac{1}{\rho}\right)^2\left( \dfrac{2\gamma}{M_2 + \frac{1}{\rho}} - \gamma^2 \right), 1 - \lambda\rho_{\min} (BB^T) \right).
\end{equation}
\end{theorem}
\begin{proof}
See Appendix~\ref{Appendix_samples bound lemma}.

\end{proof}

By simple computation we have that $\eta\in[0, 1)$. Since $K$ is the number of iterations in subproblems and is independent of $\eta$,  
when $K\rightarrow +\infty$, we have that $\eta^K \rightarrow 0$.

Thus the gradients $\nabla U_\rho(\theta_n)  $ are almost accurate and the inequality \eqref{ineq:upper_bound_Expectation} is reduced to 
\begin{equation}
     \mathbb{E} \left( U_\rho\left( \theta_{n} \right) - U_\rho\left( x^* \right)  \right) \leqslant \left(1 - m_\rho\delta \right)^n  \mathbb{E}\left( U_\rho(\theta_0) - U_\rho(x^*) \right) + \dfrac{ d    }{ \rho  m_\rho },
\end{equation}
which matches Proposition 1 in \cite{dalalyan2017theoretical}. This lemma implies that the upper bound of  $\mathbb{E}U_\rho(\theta_n)$ includes a term not depending on the discretization parameter $\delta$ and another term approaching to zero as $n\rightarrow +\infty$. Moreover, we can also obtain the upper bound of $\mathbb{E}\left\| \theta_n - x^* \right\|_2^2$ by the $m_\rho$-strongly convexity of $U_\rho$ and $\frac{m_\rho}{2} \left\| x - x^* \right\|_2^2 \leqslant  U_\rho(x) - U_\rho(x^*)   $. Both the boundness of $\mathbb{E}U_\rho(\theta_n)$ and $\mathbb{E}\left( \left\| \theta_n - x^* \right\|_2^2 \right)$ essentially require the strongly convexity of $U_\rho$. 

Theorem \ref{samples bound lemma} shows that for any $K\in \mathbb{N}$, Algorithm \ref{ULA by PDFP} will not blow up in the sense of expectation. The remaining portion of this section will complete the nonasymptotic error analysis of the sampling. We now present a lemma quantifying the accummulated gradients of the log-density and the accummulated errors:
\begin{lemma}
Under the conditions in Lemma \ref{basic lemma for convergence}, we have
\label{gradient norm sum lemma}
\begin{equation}
\footnotesize
\begin{aligned}
    & \delta \sum_{n = 0}^{N-1}  \mathbb{E}\left\| \nabla U_\rho(\theta_n) \right\|_2^2 \leqslant \dfrac{2}{1 - \eta^K}\mathbb{E} \left( U_\rho\left( \theta_{0} \right) - U_\rho\left( x^* \right)  \right) + \dfrac{N\delta \left(2d \lambda\rho + \gamma^2C^2\eta^K\right)}{\lambda\rho^2(1 - \eta^K)}\\
    & \delta \sum_{n = 0}^{N-1} \mathbb{E}\left\| \dfrac{T_{n, 2}T_{n}^{K-1}(0, \theta_n) - \operatorname{prox}_{\rho U}(\theta_n)}{\rho} \right\|_2^2 \leqslant  \dfrac{2\eta^K}{1 - \eta^K}\mathbb{E} \left( U_\rho\left( \theta_{0} \right) - U_\rho\left( x^* \right)  \right) + \dfrac{N\delta\eta^K \left(2d \lambda\rho + \gamma^2C^2 \right)}{\lambda\rho^2(1 - \eta^K)}  \\
    & \delta \sum_{n = 0}^{N-1} \mathbb{E}\left\| \dfrac{\theta_n - T_{n, 2}T_{n}^{K-1}(0, \theta_n) }{\rho} \right\|_2^2  \leqslant \dfrac{4(1 + \eta^K) }{1 - \eta^K}\mathbb{E} \left( U_\rho\left( \theta_{0} \right) - U_\rho\left( x^* \right)  \right) + \dfrac{4N\delta  \left(d \lambda\rho(1 + \eta^K) + \gamma^2C^2\eta^K \right)}{\lambda\rho^2(1 - \eta^K)} ,
\end{aligned}
\end{equation}
where 
\begin{equation}
    \eta  = \max\left( 1 -\left(m+ \dfrac{1}{\rho}\right)^2\left( \dfrac{2\gamma}{M_2 + \frac{1}{\rho}} - \gamma^2 \right), 1 - \lambda\rho_{\min} (BB^T) \right).
\end{equation}
\end{lemma}
\begin{proof}
See Appendix \ref{Appendix_gradient norm sum lemma}. 
\\
\end{proof}

Assume that $\{\mathbf{L}_t^\rho,t\geqslant 0 \}$ is the solution of Langevin diffusion (\ref{Langevin diffusion Moreau}). For a fixed time interval $[0, l]$ where $l = N\delta$, Lemma \ref{gradient norm sum lemma} shows an upper bound of $\delta \sum_{n = 0}^{N-1}  \mathbb{E}\left\| \nabla U_\rho(\theta_n) \right\|_2^2$ when $N\rightarrow +\infty$. For the sampling error analysis we aim to prove that the solution $L_l^\rho\rightarrow\pi_\rho$ as $l\rightarrow+\infty$, and then with fixed $l$ the distribution of the $N$-th sample $\theta_N$  can be arbitrarily close to $L_l^\rho$ as $N\rightarrow+\infty$ and $K\rightarrow+\infty$.

For the samples $\{\theta_n\}_{n = 0}^{N}$ generated by Algorithm \ref{ULA by PDFP}, we introduce a continuous time Markov process $\{ \mathbf{D}_{t}:t\geqslant 0 \}$  such that the distribution of  $\left( \theta_0, \theta_1,\dots, \theta_N  \right)$ and $(\mathbf{D}_0, \mathbf{D}_\delta,\dots,\mathbf{D}_{N\delta})$ coincide. The process $\{  \mathbf{D}_t:t\geqslant 0\}  $ is defined as the solution of the stochastic differential equation
\begin{equation}
\label{D by SDE}  
     \mathrm{d} \mathbf{D}_{t}=\mathbf{b}_{t}(\mathbf{D_t}) \mathrm{d} t+\sqrt{2} \mathrm{~d} \mathbf{W}_{t}, \quad t \geqslant 0, \mathbf{D}_{0}=\theta_0,
\end{equation}
\begin{equation}
\label{definition of b} 
    \mathbf{b}_{t}(\mathbf{D}_t)= \sum_{n=0}^{\infty} \dfrac{T_{n, 2}T_{n}^{K-1}(0, \mathbf{D}_{n\delta} ) - \mathbf{D}_{n\delta}}{\rho}\mathds{1}_{[n \delta,(n + 1) \delta]}(t)  ,
\end{equation}
where $T_{n}(v,x)$ and $T_{n, 2}(v,x)$ are defined by (\ref{definition of Tn}).

\begin{theorem}
\label{Markov process induction lemma}
If the continuous time Markov process $\{ \mathbf{D}_{t}:t\geqslant 0 \}$ is defined by (\ref{D by SDE}, \ref{definition of b}), then  $\left( \theta_0, \theta_1,\dots, \theta_N  \right)$ has the same distribution with $(\mathbf{D}_0, \mathbf{D}_\delta,\dots,\mathbf{D}_{N\delta})$.
\end{theorem}
\begin{proof}
We prove this by induction.\\
For $n\in \mathbb{Z}$, assume that $ \mathbf{D}_{n\delta}$ and $\theta_n $ have the same distribution. By (\ref{D by SDE}, \ref{definition of b}) we have
\begin{equation}
\label{D induction}
\begin{aligned}
\mathbf{D}_{(n + 1)\delta} &= \mathbf{D}_{n\delta} + \int _{n \delta}^{(n + 1)\delta} b_\tau(\mathbf{D}_\tau) \mathrm{d}\tau + \int _{n \delta}^{(n + 1)\delta} \sqrt{2}\mathrm{d} \mathbf{W}_{\tau}\\
& = \mathbf{D}_{n\delta} + \int _{n \delta}^{(n + 1)\delta} \sum_{k=0}^{\infty} \dfrac{T_{n, 2}T_{n}^{K-1}(0, \mathbf{D}_{k\delta} ) - \mathbf{D}_{k\delta}}{\rho}\mathds{1}_{[k \delta,(k+1) \delta]}(\tau) \mathrm{d}\tau + \sqrt{2\delta} \xi_n \\
& = \mathbf{D}_{n\delta} + \dfrac{\delta}{\rho}\left(T_{n, 2}T_{n}^{K-1}(0, \mathbf{D}_{n\delta} ) - \mathbf{D}_{n\delta} \right)  + \sqrt{2\delta} \xi_n \\
&= \left(1 - \dfrac{\delta}{\rho}\right)\mathbf{D}_{n\delta} +  \dfrac{\delta}{\rho}T_{n, 2}T_{n}^{K-1}(0, \mathbf{D}_{n\delta}) + \sqrt{2\delta} \xi_n .
\end{aligned}
\end{equation}
Compare (\ref{D induction}) with (\ref{ULA by PDFP simplified notation}) and we can deduce that $\mathbf{D}_{(n + 1)\delta}$ and $ \theta_{n+1} $ have the same distribution. By induction we complete the proof.

\end{proof}

Now we have a continuous time Markov process $\{ \mathbf{D}_{t}:t\geqslant 0 \}$. To obtain the KL distance between the distributions of the processes $\{\mathbf{L}^\rho: t \in[0,  N\delta]\}$ and $\{\mathbf{D}: t \in[0,  N\delta]\}$ we use a lemma from \cite{dalalyan2017theoretical} based on the Girsanov formula:

\begin{lemma}
\label{KL divergence lemma}
If for some $B>0$ the non-anticipative drift function $\mathbf{b}: C\left(\mathbb{R}_{+}, \mathbb{R}^{d}\right) \times$ $\mathbb{R}_{+} \rightarrow \mathbb{R}^{d}$ satisfies the inequality $\|\mathbf{b}(\mathbf{D}, t)\|_{2} \leqslant B\left(1+\|\mathbf{D}\|_{\infty}\right)$ for every $t \in[0, N\delta]$ and every $\mathbf{D} \in$
$C\left(\mathbb{R}_{+}, \mathbb{R}^{d}\right)$, then the Kullback-Leibler divergence between $\mathbb{P}_{\mathbf{L}^\rho}^{\mathbf{x}, N\delta }$ and $\mathbb{P}_{\mathbf{D}}^{\mathbf{x}, N\delta}$, the distributions of the processes $\{\mathbf{L}^\rho: t \in[0,  N\delta]\}$ and $\{\mathbf{D}: t \in[0,  N\delta]\}$ with the initial value $\mathbf{L}^\rho_{0}=\mathbf{D}_{0}=\mathbf{x}$, is given by
\begin{equation}
    \mathrm{KL}\left(\mathbb{P}_{\mathbf{L}^\rho}^{\mathbf{x},  N\delta} \| \mathbb{P}_{\mathbf{D}}^{\mathbf{x},  N\delta}\right)\leqslant\frac{1}{4} \int_{0}^{ N\delta} \mathbb{E}\left[\left\|\nabla U_\rho\left(\mathbf{D}_{t}\right)+\mathbf{b}_{t}(\mathbf{D}_t)\right\|_{2}^{2}\right] \mathrm{d} t.
\end{equation}
\end{lemma}
Using lemma \ref{KL divergence lemma} we can prove the following theorem which gives an upper bound of the KL divergence:
\begin{theorem}
\label{KL divergence error theorem}
Let $l = N\delta$ be fixed. Assume that $\mathbf{D}$ is defined by (\ref{D by SDE}, \ref{definition of b}). Suppose that all the conditions of Lemma \ref{basic lemma for convergence} and Lemma \ref{KL divergence lemma} are satisfied,  then
\begin{equation}
    \begin{aligned}
    &\operatorname{KL}\left(\mathbb{P}_{\mathbf{L}^\rho}^{\mathbf{x}, l} \| \mathbb{P}_{\mathbf{D}}^{\mathbf{x}, l}\right) \leqslant \dfrac{2\delta^2(1 + \eta^K) + 3\rho^2\eta^K}{3\rho^2(1 - \eta^K)} \mathbb{E} \left( U_\rho\left( x \right) - U_\rho\left( x^* \right)  \right)\\
    & \quad + \dfrac{ld\lambda\rho\left(4\delta^2(1+\eta^K) + 3\delta\rho(1 - \eta^K) + 6 \rho^2\eta^K\right) + l\gamma^2C^2\eta^K(4\delta^2 + 3 \rho^2)  }{6\lambda\rho^4(1 - \eta^K)},
    \end{aligned}
\end{equation}
where 
\begin{equation}
    \eta  = \max\left( 1 -\left(m+ \dfrac{1}{\rho}\right)^2\left( \dfrac{2\gamma}{M_2 + \frac{1}{\rho}} - \gamma^2 \right), 1 - \lambda\rho_{\min} (BB^T) \right).
\end{equation}
\end{theorem}

\begin{proof}
See Appendix \ref{Appendix_KL divergence error theorem}.

\end{proof}

Given fixed $\rho$, this upper bound of  $\operatorname{KL}\left(\mathbb{P}_{\mathbf{L}^\rho}^{\mathbf{x}, l} \| \mathbb{P}_{\mathbf{D}}^{\mathbf{x}, l}\right)$ tends to $0$ as $\delta\rightarrow 0$ and $K\rightarrow+\infty$. Meanwhile, this upper bound also partly depends on the initial sample $\mathbf{D}_0 = \theta_0 = x$. Up to now, we have no detailed assumption on $ \theta_0$. If $\theta_0$ is drawn from the initial distribution $\nu$, from lemma \ref{Moreau lemma} and lemma 1 in \cite{dalalyan2017theoretical} one can deduce the following lemma:
\begin{lemma}
\label{pi_rho exponentially fast lemma}
$m\in \mathbb{R}, m>0$. If $U$ is $m$-strongly convex and $m_\rho = \frac{m}{1+\rho m}$, then for any initial probability density $\nu$ we have
\[
\left\|  \nu \mathbf{P}_{\mathbf{L}^\rho}^t - \pi_\rho  \right\|_{\mathrm{TV}} \leqslant \dfrac{1}{2} \chi ^2(\nu\| \pi_\rho)^{1/2} \exp\left(\dfrac{-tm_\rho}{2}\right), \quad \forall t \geqslant 0.
\]
\end{lemma}
\begin{proof}
See lemma 1 in \cite{dalalyan2017theoretical}.

\end{proof}

We can prove the following lemma when the initial distribution $\nu$ is a Gaussian distribution $\mathcal{N}_d (x^*, \rho \mathbf{I}_d)$ with mean $x^*$.

\begin{lemma}
\label{X2 divergence lemma}
$m\in \mathbb{R}, m>0$. Assume that  $x^*$ is the optimal solution of problem (\ref{MAP}). If $U$ is $m$-strongly convex and $m_\rho = \frac{m}{1+\rho m}$, if $\nu$ is the density of the Gaussian distribution $\mathcal{N}_d (x^*, \rho \mathbf{I}_d)$, then we have
\[
\int _{\mathbb{R}^{d}}  \dfrac{\nu^2(x)}{\pi_\rho(x)}\mathrm{d} x \leqslant  \dfrac{1}{(\rho m_\rho)^{\frac{d}{2}}}.
\]
\end{lemma}
\begin{proof}
The proof follows the same pattern of lemma 5 in \cite{dalalyan2017theoretical}.

From  (\ref{Moreau approximation 2}), Lemma \ref{Moreau lemma} and Lemma \ref{Lipschitz gradient lemma},
\begin{equation}
\small
\begin{aligned}
&\pi_{\rho}(x)^{-1} =\exp \{U_\rho(x)\} \int_{\mathbb{R}^{d}} \exp \{-U_\rho(\overline{x})\} \mathrm{d} \overline{x} =\exp \{U_\rho(x)-U_\rho(x^*)\} \int_{\mathbb{R}^{d}} \exp \{-U_\rho(\overline{x})+U_\rho(x^*)\} \mathrm{d} \overline{x} \\
& \leqslant \exp \left\{\nabla U_\rho(x^*)^{\mathrm{T}}(x-x^*)+\frac{1}{2\rho}\|x-x^*\|_{2}^{2}\right\} \int_{\mathbb{R}^{d}} \exp \left\{-\nabla U_\rho(x^*)^{\mathrm{T}}(\overline{x}-x^*)-\frac{m_\rho}{2}\|\overline{x}-x^*\|_{2}^{2}\right\} \mathrm{d} \overline{x} \\
& = \left(\frac{2 \pi}{m_\rho}\right)^{d / 2} \exp \left( \frac{1}{2\rho}\|x-x^*\|_{2}^{2} \right),
\end{aligned}
\end{equation}
then we have
\begin{equation}
\begin{aligned}
\int _{\mathbb{R}^{d}}  \dfrac{\nu^2(x)}{\pi_\rho(x)}\mathrm{d} x &=(2 \pi \rho)^{-d} \int_{\mathbb{R}^{d}} \exp \left\{-\frac{1}{\rho}\|x - x^* \|_{2}^{2}\right\} \pi_\rho(x)^{-1} \mathrm{d} x \\
& \leqslant(2 \pi \rho)^{-d}\left(\frac{2 \pi}{m_\rho}\right)^{d / 2}  \int_{\mathbb{R}^{d}} \exp \left\{-\frac{\|x-x^*\|_{2}^{2}}{2 \rho}\right\} \mathrm{d} x \\
& = \dfrac{1}{(\rho m_\rho)^{\frac{d}{2}}}.
\end{aligned}
\end{equation}

\end{proof}

In the next theorem we finally give the error analysis of  the Total-Variation norm between the distribution of the $N$-th sample $\theta_n$ and $\pi_\rho$.

\begin{theorem}
\label{TV norm error theorem}
Let $l = N\delta$. Assume that $\mathbf{D}$ is defined by (\ref{D by SDE}, \ref{definition of b}). Suppose that all the conditions of Lemma \ref{basic lemma for convergence} and Lemma \ref{KL divergence lemma} are satisfied. Assume that $\nu$ is the  Gaussian distribution $\mathcal{N}_d (x^*, \rho \mathbf{I}_d)$. If $m>0$ and $m_\rho = \frac{m}{1+\rho m}$,  then the TV-norm between the distribution of the $N$-th sample $\theta_N$ and the distribution $\pi_\rho$ satisfies
\begin{equation}
\small
\begin{aligned}
    &\left\|\nu \mathbf{P}_{\theta_N}-\mathbf{P}_{\pi_\rho}\right\|_{\mathrm{TV}}\leqslant  \dfrac{1}{2}\exp\left( -\dfrac{d}{4}\log(\rho m_\rho) - \dfrac{lm_\rho}{2}  \right)+ \\
    & \sqrt{ \dfrac{\lambda d \left( 2\delta^2\rho^2 + 4l\delta^2\rho + 3l\delta\rho^2   \right) +\eta^K\left[  \lambda d\left( 2\delta^2\rho^2 + 3\rho^4 + 4l\delta^2\rho - 3l\delta\rho^2 + 6l\rho^3  \right) + l\gamma^2C^2\left(  4\delta^2 + 3\rho^2  \right) \right]  }{12\lambda\rho^4(1 - \eta^K)}  },
\end{aligned}
\end{equation}
where 
\begin{equation}
    \eta  = \max\left( 1 -\left(m+ \dfrac{1}{\rho}\right)^2\left( \dfrac{2\gamma}{M_2 + \frac{1}{\rho}} - \gamma^2 \right), 1 - \lambda\rho_{\min} (BB^T) \right).
\end{equation}
Therefore for any fixed $\rho$, $\forall \epsilon > 0$, $\exists l>0, \delta \in (0, \rho] $ and $K\in \mathbb{N}$, such that $\left\|\nu \mathbf{P}_{\theta_N}-\mathbf{P}_{\pi_\rho}\right\|_{\mathrm{TV}}  < \epsilon$.
\end{theorem}
\begin{proof}
See Appendix \ref{Appendix_TV norm error theorem}.

\end{proof}

This upper bound demonstrates that, in order to make the error small one first needs a long burn-in time $l$. While  $l$ is large enough and remains fixed, small discretization step-size $\delta$ and more iterations $K$ will lead to a satisfactory error. This also matches Theorem 2 in \cite{dalalyan2017theoretical}.

\begin{figure}[htbp]
    \centering
    \includegraphics[width=.9\textwidth]{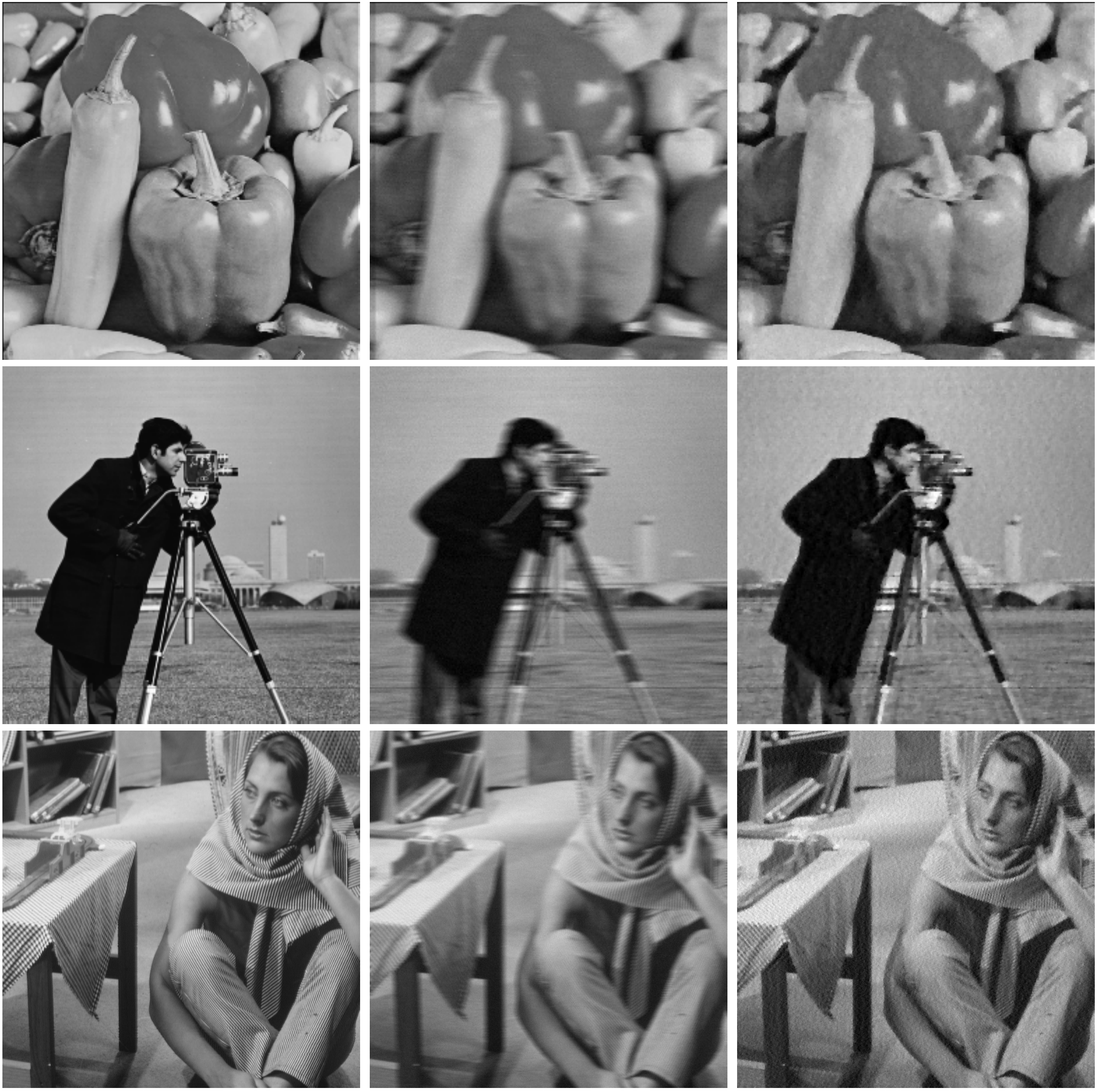}
    \caption{Left: the ground truth. Middle: the blurred and noisy image. Right: the posterior mean.}
    \label{fig:peppers}
\end{figure}

\section{Numerical experiments}\label{sec:examples}

To demonstrate the performance  of the proposed algorithms, we provide two practical examples -- an image motion deblurring problem and a computerized tomography (CT) reconstruction problem. We formulate both problems in the Bayesian framework and therefore sampling their posterior distributions
is the primary goal here.

\textbf{Choices of $K$ and $\delta$:} The stepsize $\delta$ should satisfy the upper bound studied in Section \ref{sec:convergence}. Both $\delta$ and the number of iterations $K$ control a tradeoff between asymptotic accuracy and convergence speed. For ULA-PDFP, using large $\delta$ and small $K$ then the Markov chain will move quickly to its stationary regime, ignoring a larger bias. We recommend using $K = 1$, since in later results $K=1$ leads to a satisfactory bias. However in the situations where a small bias is expected, one should choose a small $\delta$ and a large $K$, though more computation is required during the burn-in time of the Markov chains.

\subsection{Image motion deblurring}

In the image motion deblurring problem, suppose that we use the TV prior, and
 the resulting posterior distribution 
\[
\pi(\theta) \propto\exp \left(  -\dfrac{\| y - A\theta \|_2^2  }{2\sigma^2} - \lambda \| \nabla \theta \|_1  \right), 
\]
where $y(t)$ is the blurred image, $\theta(t)$ is the target image that we want to reconstruct, 
$\sigma^2$ is the observation noise variance (assuming zero-mean Gaussian noise), $\lambda$ is the regularization coefficient, 
and $A$ is a linear motion blur operator in the form of
\[(A\theta)(t) = \int \theta(\tau)K(t - \tau)\text{d}\tau.
\]
We use three commonly used tested images: Peppers, Cameraman and Barbara~(left column in Figure \ref{fig:peppers}).
In all three experiments, we choose $\sigma = 0.01$ and operator $A$ formed by the kernel $K$ of size $10 \times 10$. 
The dimensionality of the unknown images and the associated regularization parameter $\lambda$ values
are given in Table~\ref{tab:deblur1}. 
We use synthetic data (Middle column in Figure \ref{fig:peppers}) generated from the ground truth images (left column in Figure \ref{fig:peppers}). 
The posterior mean is used as an estimator of the original image. In this experiment we draw $10000$ samples
 from the posterior $\pi(
\theta)$. We use the following quantitative  measures to assess the performance of the sampling methods.
To compare the estimation error we compute the peak signal-to-noise ratio (PSNR) of the posterior sample mean, which is used as an estimator of the true image. For sampling efficiency comparison we respectively calculate the effective sample size (ESS) \cite{kass1998markov} and the expected square jumping distance (ESJD) \cite{pasarica2010adaptively} of the samples.

\begin{table}
\centering
\begin{tabular}{c|c|c|c}
& pepper&cameraman&barbara\\
\hline
dimensionality&$256\times256$&$256\times256$&$512\times512$\\
\hline
$\lambda$&0.13&0.12&0.08
\end{tabular}
\caption{Dimensionality and $\lambda$ values.}\label{tab:deblur1}
\end{table}

We first examine the unadjusted algorithms, and we restate that, 
without the Metropolis step, the samples obtained by this type of methods are subject 
to bias. 
Apart from the proposed PDFP based algorithm, we also 
implement Moreau-Yosida unadjusted Langevin algorithm (MYULA) in \cite{durmus2018efficient}. 
Note that in MYULA, 
it is proposed to accurately solve the subproblem by Chambolle2004 \cite{chambolle2004algorithm}, and to have a more comprehensive  comparison, we also implement a slightly modified
version of MYULA -- replacing Chambolle2004 with a $K$-step Chambolle-Pock \cite{chambolle2011first}. 


We summarize the results in Table \ref{tab:ULAEExperiment1},
and note that for ULA-PDFP and MYULA-CP we tested three cases $K=1,5$, and $K=100$.
In particular in the  $K = 100$ cases the subproblem is considered to be precisely solved,
and in fact our numerical experiments suggest that 
most of the  subproblems can meet the stopping criteria $\| x_{n,k+1} - x_{n,k} \|<10^{-5}$ in less than $30$ steps.
For MYULA, the subproblem is solved accurately using Chambolle2004 \cite{chambolle2004algorithm}.
From the table we observe that the PSNR and ESJD of the sample means calculated 
by all the methods are approximately the same,
suggesting that all the methods can produce similar sampling results.
Quite interestingly, the results show that PDFP and CP with $K=1$ can produce results of the same PSNR and ESJD 
as solving the subproblem accurately. 
On the other hand, as has been discussed smaller $K$ leads to less computational burden, 
which is supported by the time cost shown in the table.  
Also ULA-PDFP with $K=1$ seems to be the most efficient one in terms of time cost. 
In summary, the results suggest that while all the algorithms yield similar sampling performance,
those that do not seek to solve the subproblem accurately are significantly more computationally efficient. 

\begin{table}[!ht]
\centering
\scriptsize
\begin{tabular}{|l|r|r|r|r|r|r|r|r|r|r|}
\hline
\multicolumn{2}{|l}{$\rho = 0.01$}& \multicolumn{3}{|c}{peppers}& \multicolumn{3}{|c}{cameraman}&\multicolumn{3}{|c|}{barbara}          \\
\hline
                              &K  &PSNR      &ESJD &time        &PSNR       &ESJD  &time          &PSNR       &ESJD       &time         \\ 
\hline
ULA-PDFP                    &1  &26.48     &1311 &\textbf{55s}&24.13      &1311  &\textbf{59s}  &23.20      &5243       &\textbf{246s}\\ 
ULA-PDFP                    &5  &26.50     &1311 &191s        &24.18      &1311  &195s          &23.20      &5243       &842s         \\ 
ULA-PDFP                    &100&26.42     &1311 &242s        &24.18      &1311  &267s          &23.22      &5243       &1047s        \\ 
\hline
MYULA-CP                    &1  &26.44     &1311 &64s         &24.17      &1310  &66s           &23.18      &5239       &287s         \\ 
MYULA-CP                    &5  &26.49     &1310 &146s        &24.11      &1309  &137s          &23.21      &5237       &656s         \\ 
MYULA-CP                    &100&26.46     &1310 &1097s       &24.16      &1310  &980s          &23.21      &5238       &4133s        \\ 
\hline
MYULA                       &100&26.43     &1310 &551s        &24.17      &1310  &525s          &23.22      &5238       &2421s        \\ 
\hline
\end{tabular}
\caption{Comparison of the unadjusted Langevin algorithms.}
\label{tab:ULAEExperiment1}
\end{table}


Next we test the algorithms with the  additional Metropolis (accept-reject) step included.
More precisely we implement the following algorithms: MALA with subgradient,
the PMALA method in \cite{pereyra2016proximal}, a variant of PMALA with Chambolle2004 replaced by 
$K$-step Chambolle-Pock,
and the proposed PDFP based algorithm denoted as MALA-PDFP. 
 The results of all the methods are compared in Table \ref{tab:MALAPeppers},
 and we reinstate that thanks to the Metropolis step, 
 the samples are asymptotically unbiased. 
For the stability of PMALA and MALA-PDFP, step size $\delta$ should be no larger than parameter $\rho$.  Following \cite{pereyra2016proximal} we fix $\delta = \rho$ and the values of them (that are shown in Table \ref{tab:MALAPeppers})
are chosen such that  the acceptance rates of all the algorithms are around $50\%$ \cite{robert1999monte, roberts2002langevin}
for fair comparison.
First we have found that MALA with subgradient clearly has the worst performance among all the methods, 
a finding agreeing with \cite{pereyra2016proximal}. Moreover, 
 in both MALA-PDFP and PMALA-CP, we can see that the results
 of $K=5$ are rather close to those of $K=100$ and PMALA where in both cases the subproblem is solved accurately. Notably in Table \ref{tab:MALAPeppers} the run time of MALA-PDFP for $K = 100$ is similar or less than that for $K = 5$, this is because in this experiments $\rho$ is much smaller than Table \ref{tab:ULAEExperiment1} and the stopping criteria $\| x_{n,k+1} - x_{n,k} \|<10^{-5}$ is met even $k<5$.  
More interestingly, however, PMALA-CP with $K=1$ yields substantially worse results (in terms of ESS and ESJD) than 
the algorithms that solve the subprobem accurately, while MALA-PDFP with $K=1$ produces results that are comparable 
to those. While this is an interesting indicator that the 1-step MALA-PDFP may be an effective and efficient sampling algorithm,  further investigation and more comprehensive tests of the method are needed.

 

\begin{table}[ht!]
\centering
\begin{tabular}{|l|r|r|r|r|r|r|}
\hline
             & K &PSNR       &ESJD       &ESS       &parameters             &time  \\
             \hline
             pepper & \multicolumn{6}{|l|}{ } \\
\hline
MALA(subgradient)   &   &25.58      &3.9        &4.03           &$\delta =$ 8e-5        & \textbf{93s}  \\
\hline
PMALA-CP          &1  &26.05      &19.4       &4.09           &$\rho = \delta =$ 3e-4 &103s  \\
PMALA-CP          &5  &26.68      &439.0      &4.78           &$\rho = \delta =$ 7e-3 &239s  \\
PMALA-CP          &100&26.70      &427.7      &4.75           &$\rho = \delta =$ 7e-3 &1218s \\
\hline 
PMALA             &100&26.69      &420.9      &4.76           &$\rho = \delta =$ 7e-3 &581s  \\
\hline
MALA-PDFP         &1  &26.61      &441.0      &4.81           &$\rho = \delta =$ 7e-3 &108s  \\
MALA-PDFP         &5  &26.66      &439.8      &4.78           &$\rho = \delta =$ 7e-3 &257s  \\
MALA-PDFP         &100&26.70      &437.6      &4.76           &$\rho = \delta =$ 7e-3 &295s  \\
\hline
cameraman          & \multicolumn{6}{|l|}{ } \\
\hline
MALA(subgradient)  &   &23.65      &3.4        &3.97           &$\delta =$  6e-5      &\textbf{89s}   \\
\hline
PMALA-CP         &1  &24.31      &18.2       &4.05           &$\rho=\delta=$ 4e-4   &107s  \\
PMALA-CP         &5  &24.46      &384.6      &4.68           &$\rho=\delta=$ 6e-3   &179s  \\
PMALA-CP         &100&24.51      &390.9      &4.70           &$\rho=\delta=$ 6e-3   &877s  \\
\hline 
PMALA            &100&24.54      &383.3      &4.65           &$\rho=\delta=$ 6e-3   &442s  \\
\hline
MALA-PDFP        &1  &24.51      &370.7      &4.62           &$\rho=\delta=$ 6e-3   &91s   \\
MALA-PDFP        &5  &24.57      &384.1      &4.67           &$\rho=\delta=$ 6e-3   &230s  \\
MALA-PDFP        &100&24.58      &375.1      &4.66           &$\rho=\delta=$ 6e-3   &234s  \\
\hline
barbara             & \multicolumn{6}{|l|}{ }  \\
\hline
MALA(subgradient)   &   &22.09      &11.3       &3.99           &$\delta =$  5e-5     &\textbf{338s}  \\
\hline
PMALA-CP          &1  &23.11      &47.4       &3.96           &$\rho=\delta=$ 2e-4  &403s  \\
PMALA-CP          &5  &23.28      &1073.8     &4.30           &$\rho=\delta=$ 5e-3  &790s  \\
PMALA-CP          &100&23.23      &993.1      &4.29           &$\rho=\delta=$ 5e-3  &3827s \\
\hline 
PMALA             &100&23.29      &947.2      &4.27           &$\rho=\delta=$ 5e-3  &1793s \\
\hline
MALA-PDFP         &1  &23.30      &934.7      &4.24           &$\rho=\delta=$ 5e-3  &387s  \\
MALA-PDFP         &5  &23.24      &1033.2     &4.30           &$\rho=\delta=$ 5e-3  &978s  \\
MALA-PDFP         &100&23.28      &973.4      &4.28           &$\rho=\delta=$ 5e-3  &865s  \\
\hline
\end{tabular}
\caption{Comparison of the Metropolis-adjusted Langevin algorithms.}
\label{tab:MALAPeppers}
\end{table}

\subsection{Computed Tomography reconstruction of medical image}

In this section we consider the computed tomography (CT) reconstruction problem
with the posterior distribution
\[
\pi(\theta) \propto\exp \left(  -\dfrac{\| y - A\theta \|_2^2  }{2\sigma^2} - \lambda \| \nabla \theta \|_1  \right),
\]
where $\theta\in \mathbb{R}^{256\times256}$ is the unknown XCAT phantom image and  $y\in\mathbb{R}^{512\times 90}$ is the projection observed. The range of $y$ is about $[0, 5.0]^{512\times90}$. 
The observation noise is assumed to be additive white Gaussian noise with standard variance $\sigma = 0.5$
and $\lambda$ is taken to be $50$. The operator $A$ is the Radon transform which can be efficiently computed by a parallelizable algorithm in \cite{gao2012fast} using fan-beam geometry, but still very time-consuming that less calls of $A$ will significantly reduce the time cost. 
In this experiment the number of detectors is $512$ and that of the viewers is $90$
defining a highly ill-posed problem. 

\begin{figure}[ht]
    \centering
    \includegraphics[width=0.6\hsize]{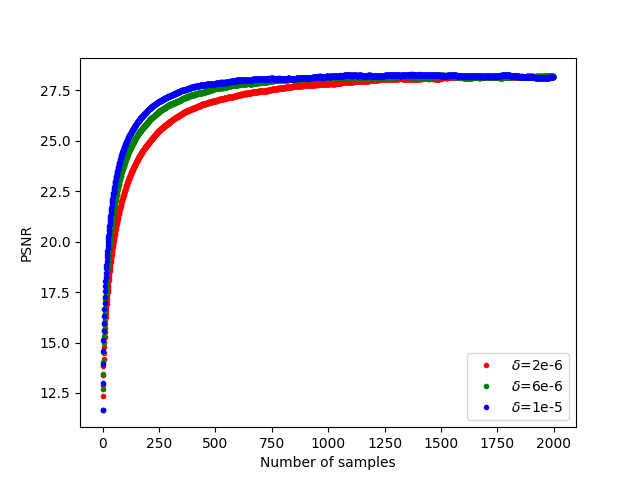}
    \caption{PSNR of the samples (ULA-PDFP) in the burn-in period.}
    \label{fig:PSNR}
\end{figure}

Before the Markov chain reaches its stable regime, the burn-in time takes less than 2000 samples. Smaller stepsize $\delta$ leads to longer burn-in period as shown in Fig. \ref{fig:PSNR}. After the burn-in time, we draw $10000$ samples  from 
the posterior $\pi$ with the same set of unadjusted algorithms in the first example,
and show the results in Table \ref{tab:ULAPhantom}. In all the algorithms we use $\rho=10^{-5}$.
The results in this examples are largely consistent with those reported in the first example:
all the methods produce similar results in terms of PSNR while those with small $K$ are more 
computationally efficient. 
Next we test the Metropolis-adjusted algorithms -- again by drawing $10,000$ samples from
the posterior, and the results are shown in Table \ref{tab:MALAPhantom}. Once again the parameters values are chosen so that the acceptance probability is around $50\%$ \cite{robert1999monte, roberts2002langevin}.  
We observe that in this example the 1-step MALA-PDFP 
has similar performance as the algorithms that solve the subproblem accurately, while 
1-step PMALA-CP is clearly less efficient in terms of both ESS and ESJD,
supporting our results in the first example.

\begin{figure}[ht]
    \centering
    \includegraphics[width=1\hsize]{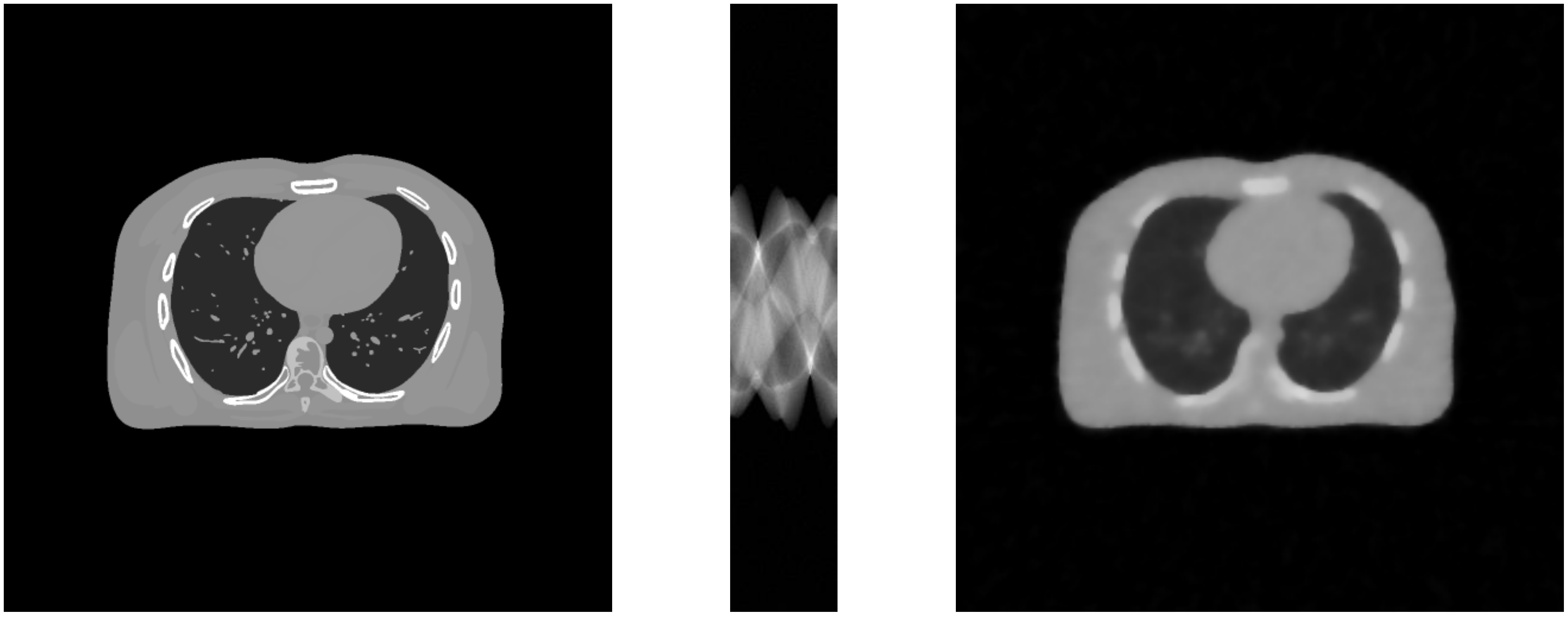}
    \caption{Left: Original image (unknown). Mid: Observation ($512\times 90 $). Right: Reconstructed image (posterior mean).}
    \label{fig:phantom}
\end{figure}

\begin{table}
\centering
\begin{tabular}{|l|r|r|r|}

\hline
                      &  PSNR   & K   &   time   \\
\hline
ULA-PDFP              &  29.22  & 1   &   \textbf{72s}    \\
ULA-PDFP              &  29.26  & 5   &   275s   \\
ULA-PDFP              &  29.26  & 100 &   1513s   \\
\hline 
MYULA-CP              &  29.25  & 1   &   80s   \\
MYULA-CP              &  29.26  & 5   &   166s   \\
MYULA-CP              &  29.24  & 100 &   720s   \\
\hline
MYULA                 &  29.26  & 100 &   1615s   \\
\hline
\end{tabular}
\caption{Comparison of the unadjusted Langevin algorithms.}
\label{tab:ULAPhantom}
\end{table}

\begin{table}[ht!]
\centering
\begin{tabular}{|l|r|r|r|r|r|r|r|}
\hline
                       &  PSNR      & ESJD   & ESS~mean  &  parameters               &  K  &   time  \\
\hline 
PMALA-CP               &  28.30     & 5.6e-4 & 4.03      & $\rho = \delta =$ 1.0e-8  & 1   &   \textbf{98s} \\
PMALA-CP               &  28.51     & 5.8e-3 & 4.87      & $\rho = \delta =$ 1.0e-7  & 5   &   197s  \\
PMALA-CP               &  28.89     & 4.6e-2 & 8.44      & $\rho = \delta =$ 8.0e-7  & 100 &   579s  \\
\hline
PMALA                  &  28.91     & 4.8e-2 & 8.63      & $\rho = \delta =$ 8.0e-7  & 100 &   305s  \\
\hline
MALA-PDFP              &  28.86     & 4.9e-2 & 8.75      & $\rho = \delta=$  8.0e-7  & 1   &  120s  \\
MALA-PDFP              &  28.85     & 4.8e-2 & 8.67      & $\rho=\delta=$    8.0e-7  & 5   &  326s  \\
MALA-PDFP              &  28.86     & 4.8e-2 & 8.66      & $\rho=\delta=$    8.0e-7  & 100 &  352s  \\
\hline
\end{tabular}
\caption{Comparison of the Metropolis-adjusted Langevin algorithms.}
\label{tab:MALAPhantom}
\end{table}

\section{Conclusion}\label{sec:conclusions}

Langevin algorithms are important tools for sampling  posterior distributions 
in Bayesian inference. Since the gradient information is typically needed in the
Langevin algorithms, it is particularly challenging to apply
them to non-smooth distributions. In this work we consider 
the class of methods where one solves a proximity subproblem in each iteration.
In particular we propose to  solves the proximity subproblem  
with the PDFP algorithm,
and more importantly the method only seeks to find an approximate solution of the subproblem 
by conducting a (small) fixed number of PDFP iterations. 
We provide error analysis of the approximate PDFP based algorithms.
Our numerical experiments also suggest that the $1$-step PDFP based algorithms,
especially the Metropolis-adjusted version, yields a good performance, in terms of sampling efficiency and computation time.

\appendix

\section{Proofs}\label{sec:proofs}
\subsection{Lemma~\ref{PDFP norm lemma}}
\label{Appendix_PDFP norm lemma}

\begin{proof}
From the fixed point property by lemma \ref{PDFP fixed point lemma} we know 
\begin{equation}
    \left\{
    \begin{aligned}
    v^* &= \operatorname{prox}_{\frac{\lambda}{\gamma}g^*} \left( \dfrac{\lambda}{\gamma}B\left(x^* - \gamma  \nabla f(x^*)\right) +  (I- \lambda B B^T) v^*   \right) = \operatorname{prox}_{\frac{\lambda}{\gamma}g^*}\left(\dfrac{\lambda}{\gamma}B\phi(x^*) + Mv^*  \right)  \\
    x^* &= x^* - \gamma  \nabla f(x^*)  - \gamma B^T v^*.
    \end{aligned}
    \right.
\end{equation}
Let $ x_k, x_{k+1}, v_k, v_{k+1} $ be the variables in Algorithm \ref{PDFP for optimization}, then 
\begin{equation}
\label{norm of v - v*}
    \begin{aligned}
    & \| v_{k+1} - v^* \|_2^2  =   \left\| \operatorname{prox}_{\frac{\lambda}{\gamma}g^*}\left(\dfrac{\lambda}{\gamma} B\phi(x_k) + Mv_k \right) - \operatorname{prox}_{\frac{\lambda}{\gamma}g^*}\left(\dfrac{\lambda}{\gamma}B\phi(x^*) + Mv^*  \right)  \right\|_2^2       \\
    &\leqslant \left\langle \operatorname{prox}_{\frac{\lambda}{\gamma}g^*}\left(\dfrac{\lambda}{\gamma} B\phi(x_k) + Mv_k \right) - \operatorname{prox}_{\frac{\lambda}{\gamma}g^*}\left(\dfrac{\lambda}{\gamma}B\phi(x^*) + Mv^*  \right),\right.\\
    & \quad\quad\quad \left.\dfrac{\lambda}{\gamma} B\phi( x_k) + Mv_k - \dfrac{\lambda}{\gamma}B\phi(x^*) - Mv^* \right\rangle \\
    & = \dfrac{\lambda}{\gamma}\left\langle  v_{k+1} - v^*, B\left( \phi(x_k) - \phi(x^*) \right)   \right\rangle + \left\langle  v_{k+1} - v^*, M\left( v_k - v^* \right)   \right\rangle.
    \end{aligned}
\end{equation}
The inequality follows from the firmly nonexpansiveness of $\operatorname{prox}_{\frac{\lambda}{\gamma}g^*}(\cdot)$ (Definition \ref{nonexpansive operator}).
By the definition of $x_{k+1}$ in Algorithm \ref{PDFP for optimization},
\begin{equation}
\label{norm of x - x*}
    \begin{aligned}
    &  \left\| x_{k+1} - x^* \right\|_2^2   = \left\|   \phi(x_k) - \phi(x^*) - \gamma B^T\left(  v_{k+1} -  v^*  \right)    \right\|_2^2   \\
    & = \left\|   \phi(x_k) - \phi(x^*) \right\|_2^2 -2\gamma \left\langle \phi(x_k) - \phi(x^*),  B^T\left( v_{k+1} -   v^*  \right)  \right\rangle  + \dfrac{\gamma^2}{\lambda^2}\left\| \lambda B^T\left(  v_{k+1} - v^*  \right) \right\|_2^2 \\
    & = \left\|   \phi(x_k) - \phi(x^*) \right\|_2^2 -2\gamma \left\langle \phi(x_k) - \phi(x^*),  B^T\left( v_{k+1} -   v^*  \right)  \right\rangle + \dfrac{\gamma^2}{\lambda}\left\|   v_{k+1} -   v^*  \right\|_2^2\\
    & \quad  - \dfrac{\gamma^2}{\lambda}\left\|   v_{k+1} -   v^*   \right\|_M^2.
    \end{aligned}
\end{equation}
Here the last equality follows from the definition  $M = I - \lambda BB^T$ and $  \|  z \| _M := \sqrt{\langle z, Mz \rangle}  $.\\
Combine (\ref{norm of v - v*}) with (\ref{norm of x - x*}),
\begin{equation}
\label{norm of x and v}
    \begin{aligned}
    &\left\|  x_{k+1} - x^* \right\|_2^2 + \dfrac{\gamma^2}{\lambda} \| v_{k+1} - v^* \|_2^2 \\
    &=   \left\|   \phi(x_k) - \phi(x^*) \right\|_2^2 -2\gamma \left\langle \phi(x_k) - \phi(x^*),  B^T\left( v_{k+1} -   v^*  \right)  \right\rangle + \dfrac{2\gamma^2}{\lambda}\left\|   v_{k+1} -   v^*  \right\|_2^2 \\
    & \quad  - \dfrac{\gamma^2}{\lambda}\left\|   v_{k+1} -   v^*   \right\|_M^2  \\
    &\leqslant  \left\|   \phi(x_k) - \phi(x^*) \right\|_2^2 + \dfrac{2\gamma^2}{\lambda}\left\langle  v_{k+1} - v^*, M\left( v_k - v^* \right)   \right\rangle  -\dfrac{\gamma^2}{\lambda}\left\|   v_{k+1} -   v^*   \right\|_M^2 \\
    &\quad  -2\gamma \left\langle \phi(x_k) - \phi(x^*),  B^T\left( v_{k+1} -   v^*  \right)  \right\rangle + 2\gamma\left\langle  v_{k+1} - v^*, B\left( \phi(x_k) - \phi(x^*) \right)   \right\rangle \\
    &=  \left\|   \phi(x_k) - \phi(x^*) \right\|_2^2  + \dfrac{2\gamma^2}{\lambda}\left\langle  v_{k+1} - v^*, M\left( v_k - v^* \right)   \right\rangle -\dfrac{\gamma^2}{\lambda}\left\|   v_{k+1} -   v^*   \right\|_M^2 \\
    &=  \left\|   \phi(x_k) - \phi(x^*) \right\|_2^2  + \dfrac{\gamma^2}{\lambda} \left\| v_k - v^* \right\|_M^2 -\dfrac{\gamma^2}{\lambda}\left\|   v_{k+1} -   v_k   \right\|_M^2 \\
    &\leqslant  \eta_1^2\left\|   x_k - x^* \right\|_2^2 + \dfrac{\gamma^2}{\lambda} \left\| v_k - v^* \right\|_M^2  \\
    &\leqslant  \eta_1^2\left\|   x_k - x^* \right\|_2^2 + \dfrac{\gamma^2}{\lambda}(1 - \lambda\rho_{\min} (BB^T)) \left\| v_k - v^* \right\|_2^2.
    \end{aligned}
\end{equation}
The first inequality uses (\ref{norm of v - v*}). The second inequality follows from the condition that $\left\| \phi(x) - \phi(y) \right\|_2\leqslant \eta_1\|x - y\|_2$, $\forall x, y\in \mathbb{R}^d$. The last inequality uses the fact that $0 < \lambda \leqslant \frac{1}{\rho_{\max}(BB^T)}$ and $ 0\preceq M \preceq (1 - \lambda \rho_{\min}(BB^T))I $.  From the definition $\eta := \max\left(  \eta_1^2, 1 - \lambda\rho_{\min} (BB^T)\right)$, obviously  $0\leqslant \eta < 1$ since $  \eta_1^2 < 1$ and $0\leqslant 1 - \lambda\rho_{\min} (BB^T) < 1$. Then from (\ref{norm of x and v}),
\begin{equation}
\begin{aligned}
\left\|  x_{k+1} - x^* \right\|_2^2 + \dfrac{\gamma^2}{\lambda} \| v_{k+1} - v^* \|_2^2 \leqslant \eta \left(\left\|  x_{k} - x^* \right\|_2^2 + \dfrac{\gamma^2}{\lambda} \| v_{k} - v^* \|_2^2\right)\\
\Rightarrow \left\| x_k - x^* \right\|_2^2 + \dfrac{\gamma^2}{\lambda}\left\| v_k - v^* \right\|_2^2 \leqslant \eta^k \left( \left\| x_0 - x^* \right\|_2^2 + \dfrac{\gamma^2}{\lambda}\left\| v_0 - v^* \right\|_2^2 \right).
\end{aligned}    
\end{equation} 
\end{proof}

\subsection{Lemma~\ref{subproblem x norm lemma}}
\label{Appendix_subproblem x norm lemma}

\begin{proof}
Denote the optimal primal and dual solutions of the problem (\ref{PDFP for prox}) by $ x_n^* $ and $ v_n^*  $, exactly $x_n^* = \operatorname{prox}_{\rho U} \left(\theta_{n}\right) $.  Since $ \dfrac{\| x - \theta_n \|^{2}}{2\rho} + f(x) $ is $\left(m + \dfrac{1}{\rho} \right)$-strongly convex with a $\left(M_2 + \dfrac{1}{\rho} \right)$-Lipschitz gradient,  by Lemma \ref{PDFP norm lemma} we have
\begin{equation}
\label{subproblem k-step inequality}
    \left\| x_{n, k} - x_n^* \right\|_2^2 + \dfrac{\gamma^2}{\lambda}\left\| v_{n, k} - v_n^* \right\|_2^2 \leqslant \eta^k\left(  \left\| x_{n, 0} - x_n^* \right\|_2^2 + \dfrac{\gamma^2}{\lambda}\left\| v_{n, 0} - v_n^* \right\|_2^2 \right),
\end{equation}
where
\begin{equation}
\begin{aligned}
\eta 
&= \max\left( 1 -\left(m+ \dfrac{1}{\rho}\right)^2\left( \dfrac{2\gamma}{M_2 + \frac{1}{\rho}} - \gamma^2 \right), 1 - \lambda\rho_{\min} (BB^T) \right)\\
&\geqslant \max\left( 1 - \left(\dfrac{1 + \rho m}{1 +  \rho M_2}\right)^2 , 1 - \lambda\rho_{\min} (BB^T) \right).
\end{aligned}
\end{equation} 
Therefore we get
\begin{equation}
\begin{aligned}
& \left\| T_{n, 2}T_{n}^{K-1}\left(0, \theta_n\right) - \operatorname{prox}_{\rho U}(\theta_n) \right\|_2^2 = \left\| x_{n, K} - \operatorname{prox}_{\rho U} \left(\theta_{n}\right) \right\|_2^2  \\
& \leqslant  \eta^K\left(  \left\| x_{n, 0} - \operatorname{prox}_{\rho U} \left(\theta_{n}\right) \right\|_2^2 + \dfrac{\gamma^2}{\lambda}\left\| v_{n, 0} - v_n^* \right\|_2^2 \right)\\
& = \eta^K\left(  \left\| \theta_n - \operatorname{prox}_{\rho U} \left(\theta_{n}\right) \right\|_2^2 + \dfrac{\gamma^2}{\lambda} \left\| v_n^* \right\|_2^2   \right) \\
& \leqslant \eta^K\left(  \left\| \theta_n - \operatorname{prox}_{\rho U} \left(\theta_{n}\right) \right\|_2^2 + \dfrac{\gamma^2C^2}{\lambda}  \right).
\end{aligned}
\end{equation}
The second inequality follows from the fixed point lemma (\ref{PDFP fixed point lemma}) applied on problem (\ref{PDFP for prox}) that 
\begin{equation}
    v_{n}^* = \operatorname{prox}_{\frac{\lambda}{\gamma}g^*}\left( \dfrac{\lambda}{\gamma}B\left(x_{n}^* - \gamma\left( \nabla f(x_{n}^*) + \dfrac{1}{\rho}(x_{n}^* - \theta_n)\right)  -  \gamma B^T v_{n}^*\right)  + v_{n}^*  \right).
\end{equation}
Then from Lemma \ref{Moreau lemma} (4),
\begin{equation}
\begin{aligned}
& \left\| \dfrac{  T_{n, 2}T_{n}^{K-1}\left(0, \theta_n\right) - \operatorname{prox}_{\rho U}(\theta_n)}{\rho} \right\|_2^2 \leqslant \eta^K \left( \left\|  \dfrac{\theta_n - \operatorname{prox}_{\rho U} \left(\theta_{n}\right)}{\rho} \right\|_2^2 + \dfrac{\gamma^2C^2}{\lambda\rho^2} \right)\\
& = \eta^K \left( \left\| \nabla U_\rho\left( \theta_n \right) \right\|_2^2 + \dfrac{\gamma^2C^2}{\lambda\rho^2} \right).
\end{aligned}
\end{equation}

\end{proof}

\subsection{Lemma~\ref{basic lemma for convergence}}
\label{Appendix_basic lemma for convergence}
\begin{proof}
From  lemma \ref{U_rho lemma 2}, we have
\begin{equation}
\small
\label{U_rho inequality}
    \begin{aligned}
    &   U_\rho\left(x - \dfrac{\delta}{\rho}\left(x - T_{n, 2}T_{n}^{K-1}\left(v, x\right)\right)\right) - U_\rho\left(x\right)    \leqslant  -\delta\left(1 - \dfrac{\delta}{2\rho}\right)\left\| \nabla U_\rho(x)\right\|_2^2 \\ 
    &+ \dfrac{\delta^2}{2\rho}\left\| \dfrac{T_{n, 2}T_{n}^{K-1}\left(v, x\right) - \operatorname{prox}_{\rho U}(x)}{\rho} \right\|_2^2 + \delta\left( 1 - \dfrac{\delta}{\rho} \right)\left\langle \nabla U_\rho(x), \dfrac{T_{n, 2}T_{n}^{K-1}\left(v, x\right) - \operatorname{prox}_{\rho U}(x)}{\rho} \right\rangle \\
    & \leqslant - \dfrac{\delta}{2} \left\| \nabla U_\rho\left( x \right) \right\|_2^2 +  \dfrac{\delta}{2} \left\| \dfrac{ T_{n, 2}T_{n}^{K-1}\left(v, x\right) - \operatorname{prox}_{\rho U}(x)}{\rho} \right\|_2^2. 
    \end{aligned}
\end{equation}
The second inequality follows from Cauchy-Schwarz inequality since $\delta \in (0, \rho]$. \\
By (\ref{ULA by PDFP simplified notation}), (\ref{U_rho inequality}) and Lemma \ref{U_rho lemma 1},
\begin{equation}
\label{U_rho sequence inequality 1}
\begin{aligned}
& \mathbb{E} \left( U_\rho\left( \theta_{n+1} \right) - U_\rho\left( \theta_n \right)  \right)  = \mathbb{E}\left(U_\rho\left(\theta_n - \dfrac{\delta}{\rho}\left(\theta_n - T_{n, 2}T_{n}^{K-1}\left(0, \theta_n\right)\right) + \sqrt{2\delta}\xi_n \right) - U_\rho\left(\theta_n\right)  \right) \\
&\leqslant \dfrac{\delta d}{\rho} -\dfrac{\delta}{2} \mathbb{E}\left\| \nabla U_\rho\left( \theta_n \right) \right\|_2^2 +  \dfrac{\delta}{2} \mathbb{E}\left\| \dfrac{  T_{n, 2}T_{n}^{K-1}\left(0, \theta_n\right) - \operatorname{prox}_{\rho U}(\theta_n)}{\rho} \right\|_2^2.
\end{aligned}
\end{equation}
From Lemma \ref{subproblem x norm lemma} and (\ref{U_rho sequence inequality 1}),
\begin{equation} 
\begin{aligned}
      \mathbb{E} \left( U_\rho\left( \theta_{n+1} \right) - U_\rho\left( \theta_n \right)  \right) 
      &\leqslant \dfrac{\delta d}{\rho} - \dfrac{\delta}{2}(1 - \eta^K) \mathbb{E}\left\| \nabla U_\rho(\theta_n) \right\|_2^2 + \dfrac{\delta\gamma^2C^2\eta^K}{2\lambda\rho^2}\\
      & =   - \dfrac{\delta}{2}(1 - \eta^K) \mathbb{E}\left\| \nabla U_\rho(\theta_n) \right\|_2^2 + \dfrac{2\delta d \lambda\rho + \delta\gamma^2C^2\eta^K}{2\lambda\rho^2} ,
\end{aligned}
\end{equation}
where 
\begin{equation}
    \eta  = \max\left( 1 -\left(m+ \dfrac{1}{\rho}\right)^2\left( \dfrac{2\gamma}{M_2 + \frac{1}{\rho}} - \gamma^2 \right), 1 - \lambda\rho_{\min} (BB^T) \right).
\end{equation}
\end{proof}

\subsection{Theorem~\ref{samples bound lemma}}
\label{Appendix_samples bound lemma}
\begin{proof}
From Lemma \ref{Moreau lemma} (3), $U_\rho$ and $U$ have the same minimizer $x^*$. Therefore $\nabla U_\rho(x^*) = 0$. 

Let $m_\rho =\dfrac{m}{1+\rho m} $. Since $U_\rho$ is $m_\rho$-strongly convex by Lemma \ref{strongly convex Moreau envelope lemma}, it is well known \cite{boyd2004convex} that, 
\begin{equation}
\left\| \nabla U_\rho(x)\right\| _2^2 \geqslant 2m_\rho \left( U_\rho(x) - U_\rho(x^*) \right),    \quad \forall x\in \mathbb{R}^d.
\end{equation}
Together with Lemma \ref{basic lemma for convergence}, $\forall n\in \mathbb{N}$,
\begin{equation}
\begin{aligned}
\mathbb{E} \left( U_\rho\left( \theta_{n+1} \right) - U_\rho\left( x^* \right)  \right) &\leqslant \left(1 - m_\rho\delta(1 - \eta^K)\right)  \mathbb{E}\left( U_\rho(\theta_n) - U_\rho(x^*) \right) + \dfrac{2\delta d \lambda\rho +  \delta\gamma^2C^2\eta^K}{2\lambda\rho^2}\\
\Rightarrow \mathbb{E} \left( U_\rho\left( \theta_{n} \right) - U_\rho\left( x^* \right)  \right)&\leqslant \left(1 - m_\rho\delta(1 - \eta^K)\right)^n  \mathbb{E}\left( U_\rho(\theta_0) - U_\rho(x^*) \right) \\
& \quad \quad + \dfrac{2\delta d \lambda\rho +  \delta\gamma^2C^2\eta^K}{2\lambda\rho^2} \dfrac{1 - (1 - m_\rho\delta(1 - \eta^K))^n}{1 - \left(1 - m_\rho\delta(1 - \eta^K)\right)} \\
& \leqslant \left(1 - m_\rho\delta(1 - \eta^K)\right)^n  \mathbb{E}\left( U_\rho(\theta_0) - U_\rho(x^*) \right) + \dfrac{ 2d \lambda\rho +  \gamma^2C^2\eta^K}{2\lambda\rho^2 m_\rho(1 - \eta^K)}.
\end{aligned}
\end{equation} 
\end{proof}

\subsection{Lemma~\ref{gradient norm sum lemma}}
\label{Appendix_gradient norm sum lemma}
\begin{proof}
From Lemma \ref{basic lemma for convergence},
\begin{equation}
\begin{aligned}
    \dfrac{\delta}{2}(1 - \eta^K) \mathbb{E}\left\| \nabla U_\rho(\theta_n) \right\|_2^2 \leqslant \mathbb{E} \left( U_\rho\left( \theta_{n} \right) - U_\rho\left( \theta_{n+1} \right)  \right) + \dfrac{2\delta d \lambda\rho +  \delta\gamma^2C^2\eta^K}{2\lambda\rho^2},\quad n\in \mathbb{N}.
\end{aligned} 
\end{equation}
Summing the inequalities for $n = 0, 1,\dots,N-1$, we have
\begin{equation}
    \begin{aligned}
    \dfrac{\delta}{2}(1 - \eta^K) \sum_{n = 0}^{N-1}  \mathbb{E}\left\| \nabla U_\rho(\theta_n) \right\|_2^2 &\leqslant \mathbb{E} \left( U_\rho\left( \theta_{0} \right) - U_\rho\left( \theta_{N} \right)  \right) + \dfrac{N\delta \left(2d \lambda\rho +  \gamma^2C^2\eta^K\right)}{2\lambda\rho^2}  \\
    &\leqslant \mathbb{E} \left( U_\rho\left( \theta_{0} \right) - U_\rho\left( x^* \right)  \right) + \dfrac{N\delta \left(2d \lambda\rho +  \gamma^2C^2\eta^K\right)}{2\lambda\rho^2} \\
    \Rightarrow \delta \sum_{n = 0}^{N-1}  \mathbb{E}\left\| \nabla U_\rho(\theta_n) \right\|_2^2 &\leqslant \dfrac{2}{1 - \eta^K}\mathbb{E} \left( U_\rho\left( \theta_{0} \right) - U_\rho\left( x^* \right)  \right) + \dfrac{N\delta \left(2d \lambda\rho + \gamma^2C^2\eta^K\right)}{\lambda\rho^2(1 - \eta^K)}.
\end{aligned} 
\end{equation}
Then from Lemma \ref{subproblem x norm lemma},
\begin{equation}
\begin{aligned}
    & \delta \sum_{n = 0}^{N-1} \mathbb{E}\left\| \dfrac{T_{n, 2}T_{n}^{K-1}(0, \theta_n) - \operatorname{prox}_{\rho U}(\theta_n)}{\rho} \right\|_2^2 \leqslant \eta^K \delta \sum_{n = 0}^{N-1}  \mathbb{E}\left\| \nabla U_\rho(\theta_n) \right\|_2^2 + \dfrac{N\delta \gamma^2C^2\eta^K}{\lambda \rho^2}\\
    & \leqslant \dfrac{2\eta^K}{1 - \eta^K}\mathbb{E} \left( U_\rho\left( \theta_{0} \right) - U_\rho\left( x^* \right)  \right) + \dfrac{N\delta\eta^K \left(2d \lambda\rho + \gamma^2C^2\eta^K\right)}{\lambda\rho^2(1 - \eta^K)} + \dfrac{N\delta \gamma^2C^2\eta^K}{\lambda \rho^2} \\
    & =  \dfrac{2\eta^K}{1 - \eta^K}\mathbb{E} \left( U_\rho\left( \theta_{0} \right) - U_\rho\left( x^* \right)  \right) + \dfrac{N\delta\eta^K \left(2d \lambda\rho + \gamma^2C^2 \right)}{\lambda\rho^2(1 - \eta^K)} . 
\end{aligned}
\end{equation}
Using Cauchy-Schwarz inequality, 
\begin{equation}
    \begin{aligned}
        & \delta \sum_{n = 0}^{N-1} \mathbb{E}\left\| \dfrac{\theta_n - T_{n, 2}T_{n}^{K-1}(0, \theta_n) }{\rho} \right\|_2^2 = \delta \sum_{n = 0}^{N-1} \mathbb{E}\left\| \nabla U_\rho(\theta_n) -  \dfrac{T_{n, 2}T_{n}^{K-1}(0, \theta_n) - \operatorname{prox}_{\rho U}(\theta_n)}{\rho} \right\|_2^2  \\
        & \leqslant 2 \delta \sum_{n = 0}^{N-1}  \mathbb{E}\left\| \nabla U_\rho(\theta_n) \right\|_2^2 + 2 \delta \sum_{n = 0}^{N-1} \mathbb{E}\left\| \dfrac{T_{n, 2}T_{n}^{K-1}(0, \theta_n) - \operatorname{prox}_{\rho U}(\theta_n)}{\rho} \right\|_2^2 \\
        & \leqslant \dfrac{4(1 + \eta^K) }{1 - \eta^K}\mathbb{E} \left( U_\rho\left( \theta_{0} \right) - U_\rho\left( x^* \right)  \right) + \dfrac{4N\delta  \left(d \lambda\rho(1 + \eta^K) + \gamma^2C^2\eta^K \right)}{\lambda\rho^2(1 - \eta^K)}  .
    \end{aligned}
\end{equation}
\end{proof}

\subsection{Theorem~\ref{KL divergence error theorem}}
\label{Appendix_KL divergence error theorem}
\begin{proof}
According to Lemma \ref{KL divergence lemma}, 
\begin{equation}
\small
\label{KL divergence three term}
    \begin{aligned}
&\operatorname{KL}\left(\mathbb{P}_{\mathbf{L}^\rho}^{\mathbf{x}, l} \| \mathbb{P}_{\mathbf{D}}^{\mathbf{x}, l}\right) \leqslant  \frac{1}{4} \sum_{n=0}^{N-1} \int_{n \delta}^{(n + 1) \delta} \mathbb{E} \left\|\nabla U_\rho\left(\mathbf{D}_{t}\right)+\dfrac{T_{n, 2}T_{n}^{K-1}(0, \mathbf{D}_{n\delta} ) - \mathbf{D}_{n\delta}}{\rho} \right\|_{2}^{2}  \mathrm{d} t \\
&= \frac{1}{4} \sum_{n=0}^{N-1} \int_{n \delta}^{(n + 1) \delta} \mathbb{E} \left\|\nabla U_\rho\left(\mathbf{D}_{t}\right)- \nabla U_\rho(\mathbf{D}_{n\delta}) + \nabla U_\rho(\mathbf{D}_{n\delta}) +  \dfrac{T_{n, 2}T_{n}^{K-1}(0, \mathbf{D}_{n\delta} ) - \mathbf{D}_{n\delta}}{\rho}\right\|_{2}^{2}  \mathrm{d} t\\
&\leqslant  \frac{1}{2} \sum_{n=0}^{N-1} \int_{n \delta}^{(n + 1) \delta} \mathbb{E}\left[\left\|\nabla U_\rho\left(\mathbf{D}_{t}\right)- \nabla U_\rho(\mathbf{D}_{n\delta})\right\|_2^2 + \left\|\nabla U_\rho(\mathbf{D}_{n\delta}) + \dfrac{T_{n, 2}T_{n}^{K-1}(0, \mathbf{D}_{n\delta} ) - \mathbf{D}_{n\delta}}{\rho}\right\|_{2}^{2}\right] \mathrm{d} t \\
& = \frac{1}{2} \sum_{n=0}^{N-1} \int_{n \delta}^{(n + 1) \delta} \mathbb{E}\left[\left\|\nabla U_\rho\left(\mathbf{D}_{t}\right)- \nabla U_\rho(\mathbf{D}_{n\delta})\right\|_2^2 + \left\| \dfrac{T_{n, 2}T_{n}^{K-1}(0, \theta_n) - \operatorname{prox}_{\rho U}(\theta_n)}{\rho}\right\|_{2}^{2}\right] \mathrm{d} t.
\end{aligned}
\end{equation}
The last inequality follows from Cauchy-Schwarz inequality.
\\
From lemma \ref{Moreau lemma} (2), $U_\rho$ has $\frac{1}{\rho}$-Lipschitz gradient:
\begin{equation}
    \label{Lipschitz gradient of U_rho}
    \left\|\nabla U_\rho(\mathbf{D}_t) - \nabla U_\rho(\mathbf{D}_{n\delta}) \right\|_2\leqslant \dfrac{1}{\rho} \left\|\mathbf{D}_t - \mathbf{D}_{n\delta} \right\|_2.
\end{equation}
Then 
\begin{equation}
\label{inequality of nabla U_rho}
\begin{aligned}
\frac{1}{2} \sum_{n=0}^{N-1} \int_{n \delta}^{(n + 1) \delta} \mathbb{E}\left\|\nabla U_\rho\left(\mathbf{D}_{t}\right)- \nabla U_\rho(\mathbf{D}_{n\delta})\right\|_2^2 \mathrm{d}t\leqslant \frac{1}{2\rho^2}   \sum_{n=0}^{N-1} \int_{n \delta}^{(n + 1) \delta} \mathbb{E}\left\| \mathbf{D}_{t}- \mathbf{D}_{n\delta}\right\|_2^2\mathrm{d}t.
\end{aligned}
\end{equation}
From the definition of $\mathbf{D}_t$, for $t\in [n \delta, (n + 1)\delta]$, 
\begin{equation}
\begin{aligned}
\mathbf{D}_{t}- \mathbf{D}_{n\delta} &= \int _{n \delta}^{t}b_{\tau}(\mathbf{D}_{\tau})\mathrm{d}\tau + \int _{n \delta}^{t}\sqrt{2}\mathrm{d}\mathbf{W}_{\tau}\\
& = \dfrac{T_{n, 2}T_{n}^{K-1}(0, \mathbf{D}_{n\delta} ) - \mathbf{D}_{n\delta}}{\rho}\int _{n \delta}^{t} \mathds{1}_{[n \delta,(n + 1) \delta]}(\tau)\mathrm{d} \tau  + \sqrt{2}(\mathbf{W}_{t} - \mathbf{W}_{n \delta}) \\
& = \dfrac{T_{n, 2}T_{n}^{K-1}(0, \mathbf{D}_{n\delta} ) - \mathbf{D}_{n\delta}}{\rho}(t - n \delta) + \sqrt{2}(\mathbf{W}_{t} - \mathbf{W}_{n \delta}).
\end{aligned}
\end{equation}
Then
\begin{equation}
\label{integral of Dt-Dm}
    \begin{aligned}
    &\frac{1}{2\rho^2}  \sum_{n=0}^{N-1} \int_{n \delta}^{(n + 1) \delta} \mathbb{E}\left\| \mathbf{D}_{t}- \mathbf{D}_{n\delta}\right\|_2^2\mathrm{d}t \\
    & = \frac{1}{2\rho^2}  \sum_{n=0}^{N-1} \int_{n \delta}^{(n + 1) \delta} \mathbb{E}\left(\left\|  \dfrac{T_{n, 2}T_{n}^{K-1}(0, \mathbf{D}_{n\delta} ) - \mathbf{D}_{n\delta}}{\rho}(t - n \delta) \right\|_2^2+ \left\|\sqrt{2}(\mathbf{W}_{t} - \mathbf{W}_{n \delta})  \right\|_2^2\right)\mathrm{d}t \\
    & = \frac{1}{2\rho^2}  \sum_{n=0}^{N-1} \left(  \dfrac{\delta^3 }{3} \mathbb{E}\left\|  \dfrac{T_{n, 2}T_{n}^{K-1}(0, \mathbf{D}_{n\delta} ) - \mathbf{D}_{n\delta}}{\rho}\right\|_2^2 + \delta^2 d  \right) \\
    & = \frac{ \delta^3}{6\rho^2}  \sum_{n=0}^{N-1}    \mathbb{E}\left\|  \dfrac{T_{n, 2}T_{n}^{K-1}(0, \theta_n ) - \theta_n }{\rho}\right\|_2^2 + \dfrac{\delta l d }{2\rho^2}   .
    \end{aligned}
\end{equation} 
Combine (\ref{KL divergence three term}) with (\ref{inequality of nabla U_rho}, \ref{integral of Dt-Dm}), we have
\begin{equation}
\small
\begin{aligned}
&\operatorname{KL}\left(\mathbb{P}_{\mathbf{L}^\rho}^{\mathbf{x}, l} \| \mathbb{P}_{\mathbf{D}}^{\mathbf{x}, l}\right)
\leqslant \frac{ \delta^3}{6\rho^2}  \sum_{n=0}^{N-1}    \mathbb{E}\left\|  \dfrac{T_{n, 2}T_{n}^{K-1}(0, \theta_n ) - \theta_n }{\rho}\right\|_2^2 + \dfrac{\delta ld }{2\rho^2}    \\
& \quad  + \frac{1}{2} \sum_{n=0}^{N-1} \int_{n \delta}^{(n + 1) \delta} \mathbb{E}\left\| \dfrac{T_{n, 2}T_{n}^{K-1}(0, \theta_n) - \operatorname{prox}_{\rho U}(\theta_n)}{\rho} \right\|_2^2 \mathrm{d} t \\
&= \frac{ \delta^3}{6\rho^2}  \sum_{n=0}^{N-1}    \mathbb{E}\left\|  \dfrac{T_{n, 2}T_{n}^{K-1}(0, \theta_n ) - \theta_n }{\rho}\right\|_2^2 + \dfrac{\delta ld }{2\rho^2} + \frac{\delta}{2} \sum_{n=0}^{N-1}   \mathbb{E}\left\| \dfrac{T_{n, 2}T_{n}^{K-1}(0, \theta_n) - \operatorname{prox}_{\rho U}(\theta_n)}{\rho} \right\|_2^2  .
\end{aligned}
\end{equation}
Combined with Lemma \ref{gradient norm sum lemma}, we obtain the inequality
\begin{equation}
    \begin{aligned}
    &\operatorname{KL}\left(\mathbb{P}_{\mathbf{L}^\rho}^{\mathbf{x}, l} \| \mathbb{P}_{\mathbf{D}}^{\mathbf{x}, l}\right) \leqslant \dfrac{2\delta^2(1 + \eta^K) + 3\rho^2\eta^K}{3\rho^2(1 - \eta^K)} \mathbb{E} \left( U_\rho\left( x \right) - U_\rho\left( x^* \right)  \right)\\
    & \quad + \dfrac{ld\lambda\rho\left(4\delta^2(1+\eta^K) + 3\delta\rho(1 - \eta^K) + 6 \rho^2\eta^K\right) + l\gamma^2C^2\eta^K(4\delta^2 + 3 \rho^2)  }{6\lambda\rho^4(1 - \eta^K)}.
    \end{aligned}
\end{equation}
\end{proof}

\subsection{Theorem~\ref{TV norm error theorem}}
\label{Appendix_TV norm error theorem}
\begin{proof}
From triangular inequality we have
\begin{equation}
\label{TV norm triangular inequality}
    \left\|\nu \mathbf{P}_{\theta_N}-\mathbf{P}_{\pi_\rho}\right\|_{\mathrm{TV}}=\left\|\nu \mathbf{P}_{\mathbf{D}}^{N\delta}-\mathbf{P}_{\pi_\rho}\right\|_{\mathrm{TV}} \leqslant\left\|\nu \mathbf{P}_{\mathbf{L^\rho}}^{l}-\mathbf{P}_{\pi_\rho}\right\|_{\mathrm{TV}}+\left\|\nu \mathbf{P}_{\mathbf{D}}^{l}-\nu \mathbf{P}_{\mathbf{L^\rho}}^{l}\right\|_{\mathrm{TV}}.
\end{equation}
From Lemma \ref{pi_rho exponentially fast lemma} and Lemma \ref{X2 divergence lemma},
\begin{equation}
\label{TV norm inequality 1}
    \left\| \nu \mathbf{P}_{\mathbf{L}^\rho}^l - \mathbf{P}_{\pi_\rho} \right\|_{\mathrm{TV}}\leqslant \dfrac{1}{2} \chi ^2(\nu\| \pi_\rho)^{1/2} \exp\left(\dfrac{-lm_\rho}{2}\right)\leqslant \dfrac{1}{2}\exp\left( -\dfrac{d}{4}\log(\rho m_\rho) - \dfrac{lm_\rho}{2}  \right).
\end{equation}
By Pinsker inequality,
\begin{equation}
\label{TV norm inequality 2}
    \left\|\nu \mathbf{P}_{\mathbf{D}}^{l}-\nu \mathbf{P}_{\mathbf{L^\rho}}^{l}\right\|_{\mathrm{TV}} \leqslant\left\|\nu \mathbb{P}_{\mathbf{D}}^{l}-\nu \mathbb{P}_{\mathbf{L^\rho}}^{l}\right\|_{\mathrm{TV}} \leqslant\sqrt{ \frac{1}{2} \mathrm{KL}\left(\nu \mathbb{P}_{\mathbf{L^\rho}}^{l} \| \nu \mathbb{P}_{\mathbf{D}}^{l}\right) } .
\end{equation}
By Lemma \ref{KL divergence error theorem} and Lemma \ref{Lipschitz gradient lemma},
\begin{equation}
\small
    \begin{aligned}
    & \mathrm{KL} \left(\nu \mathbb{P}_{\mathbf{L^\rho}}^{l} \| \nu \mathbb{P}_{\mathbf{D}}^{l}\right)  \\
    & \leqslant\dfrac{2\delta^2(1 + \eta^K) + 3\rho^2\eta^K}{3\rho^2(1 - \eta^K)} \dfrac{d}{2} + \dfrac{ld\lambda\rho\left(4\delta^2(1+\eta^K) + 3\delta\rho(1 - \eta^K) + 6 \rho^2\eta^K\right) + l\gamma^2C^2\eta^K(4\delta^2 + 3 \rho^2)  }{6\lambda\rho^4(1 - \eta^K)} \\
    & = \dfrac{\lambda d \left( 2\delta^2\rho^2 + 4l\delta^2\rho + 3l\delta\rho^2   \right) +\eta^K\left[  \lambda d\left( 2\delta^2\rho^2 + 3\rho^4 + 4l\delta^2\rho - 3l\delta\rho^2 + 6l\rho^3  \right) + l\gamma^2C^2\left(  4\delta^2 + 3\rho^2  \right) \right]  }{6\lambda\rho^4(1 - \eta^K)} .
    \end{aligned}
\end{equation}
From (\ref{TV norm triangular inequality}, \ref{TV norm inequality 1}, \ref{TV norm inequality 2}) and above,

\begin{equation}
\small
\begin{aligned}
    &\left\|\nu \mathbf{P}_{\theta_N}-\mathbf{P}_{\pi_\rho}\right\|_{\mathrm{TV}}\leqslant  \dfrac{1}{2}\exp\left( -\dfrac{d}{4}\log(\rho m_\rho) - \dfrac{lm_\rho}{2}  \right)+ \\
    & \sqrt{ \dfrac{\lambda d \left( 2\delta^2\rho^2 + 4l\delta^2\rho + 3l\delta\rho^2   \right) +\eta^K\left[  \lambda d\left( 2\delta^2\rho^2 + 3\rho^4 + 4l\delta^2\rho - 3l\delta\rho^2 + 6l\rho^3  \right) + l\gamma^2C^2\left(  4\delta^2 + 3\rho^2  \right) \right]  }{12\lambda\rho^4(1 - \eta^K)}  },
\end{aligned}
\end{equation}
where 
\begin{equation}
    \eta  = \max\left( 1 -\left(m+ \dfrac{1}{\rho}\right)^2\left( \dfrac{2\gamma}{M_2 + \frac{1}{\rho}} - \gamma^2 \right), 1 - \lambda\rho_{\min} (BB^T) \right).
\end{equation}
\end{proof}

\bibliographystyle{plain}
\bibliography{reference}







          \end{document}